\documentclass[11pt]{article}
\usepackage{epsfig,amsmath,amssymb,graphics,verbatim,amsfonts,subfigure,psfrag}
\usepackage[letterpaper,margin=0.85in]{geometry}
\usepackage{xcolor}
\usepackage{amsthm}
\usepackage{enumitem}
\usepackage{overpic} 
\usepackage{hyperref}
\usepackage[toc,page]{appendix}

\newtheorem{thm}{Theorem}[section]

\newtheorem{lem}[thm]{Lemma}

\theoremstyle{remark}
\newtheorem{rem}[thm]{Remark}
\theoremstyle{definition}
\newtheorem{defn}[thm]{Definition}
\newtheorem{ass}[thm]{Assumptions}

\newcommand{\be}{\begin{equation}}
\newcommand{\ee}{\end{equation}}
\newcommand{\bea}{\begin{eqnarray}}
\newcommand{\eea}{\end{eqnarray}}
\newcommand{\beann}{\begin{eqnarray*}}
\newcommand{\eeann}{\end{eqnarray*}}
\newcommand{\benn}{\begin{equation*}}
\newcommand{\eenn}{\end{equation*}}

\newcommand{\cA}{{\mathcal A}}  
\newcommand{\cB}{{\mathcal B}}  
\newcommand{\cD}{{\mathcal D}}  
\newcommand{\cH}{{\mathcal H}}  
\newcommand{\cK}{{\mathcal K}}  
\newcommand{\cT}{{\mathcal T}}  
\newcommand{\cV}{{\mathcal V}}  

\def\txtd{{\textnormal{d}}}

\numberwithin{equation}{section}

\newcommand{\e}[1]{\exp\left(#1\right)}

\def\R{\mathbb{R}}

\def\I{\infty}
\begin{document}
\title{Random Attractors for Stochastic Partly \linebreak Dissipative Systems}
\author{Christian Kuehn, Alexandra Neam\c tu, Anne Pein}
\maketitle
\abstract{We prove the existence of a global random attractor for a certain class of stochastic partly dissipative systems. These systems consist of a partial (PDE) and an ordinary differential equation (ODE), where both equations are coupled and perturbed by additive white noise. The deterministic counterpart of such systems and their long-time behaviour have already been considered but there is no theory that deals with the stochastic version of partly dissipative systems in their full generality. We also provide several examples for the application of the theory.}
\section{Introduction}
In this work, we study classes of stochastic partial differential equations (SPDEs), which are 
part of the general partly dissipative system
\begin{equation}
\label{eq:SPDEintro}
\begin{array}{rcl}
\txtd u_1&=&(d\Delta u_1-h(x,u_1)-f(x,u_1,u_2))~\txtd t+B_1(x,u_1,u_2)~\txtd W_1,\\
\txtd u_2&=&(-\sigma(x)u_2-g(x,u_1,u_2)) ~\txtd t+B_2(x,u_1,u_2)~\txtd W_2,
\end{array}
\end{equation}
where $W_{1,2}$ are cylindrical Wiener processes, the 
$\sigma,f,g,h$ are given functions, $B_{1,2}$ are operator-valued, $\Delta$ is the 
Laplace operator, $d>0$ is a parameter, the equation is posed on a bounded open domain 
$D\subset \mathbb{R}^n$, $u_{1,2}=u_{1,2}(x,t)$ are the unknowns for $(x,t)\in 
D\times [0,T_{\max})$, and $T_{\max}$ is the maximal existence time. The term \textit{partly 
dissipative} highlights the fact that only the first component contains the regularizing 
Laplace operator. In this work we analyse the case of additive noise and a certain coupling, 
more precisely,  
\begin{equation}
\label{eq:basecase}
B_1(x,u_1,u_2)=B_1, \ \  B_2(x,u_1,u_2)=B_2,\ \  g(x,u_1,u_2)=g(x,u_1),
\end{equation}
where $B_{1,2}$ are bounded linear operators. A deterministic version of such a system has 
been analysed by Marion~\cite{Mar}. We are going to use certain assumptions for the 
reaction terms, which are similar to those used in~\cite{Mar}. 
The precise technical setting of our work starts in Section~\ref{sec:stochastics}. 

The goal of this work is to provide a general theory for stochastic partly dissipative 
systems and to analyse the long-time behaviour of the solution using the random dynamical 
systems approach. To this aim, we first show that the solution of our system exists 
globally-in-time, i.e.~one can take $T_{\max}=+\I$ above. Then we prove the existence of 
a pullback attractor. To our best knowledge the well-posedness and asymptotic behaviour for 
such systems (and for other coupled SPDEs and SODEs) has only been explored for special cases, 
i.e.~mainly for the FitzHugh Nagumo equation, see~\cite{BonaMast,SaSt} for solution theory 
and~\cite{Adi, Wang1, Wang2, Li} for long-time behaviour/attractor theory. Here we develop 
a much more general theory of stochastic partly dissipative systems, motivated by the numerous 
applications in the natural sciences such as the the cubic-quintic Allen-Cahn equation~\cite{KuehnEllipticCont} in elasticity. Moreover, 
unlike several previous works mentioned above, we deal with infinite-dimensional noise that 
satisfies certain regularity assumptions. These combined with the restrictions on the 
reaction terms allow us to compute sharp a-priori bounds of the solution, which are used to 
construct a random absorbing set. Even once the absorbing set has been constructed, we emphasize 
that we cannot directly apply compact embedding results to obtain the existence of an attractor. 
This issue arises due to the absence of the regularizing effect of the Laplacian in the second 
component. To overcome this obstacle, we introduce an appropriate splitting of the solution in two 
components: a regular one, and one that asymptotically tends to zero. This splitting technique goes (at least) back to Temam \cite{Temam} and it has also been 
used in the context of deterministic partly dissipative systems ~\cite{Mar} and for a 
stochastic FitzHugh-Nagumo equation with linear multiplicative noise~\cite{Wang,Zhou}. The necessary 
additional technical steps for our setting are provided in Section~\ref{subsec:compact}. Using
the a-priori bounds, we establish the existence of a pullback attractor~\cite{CraFla,Schmal,Schmal2000,Flan}; 
which has been studied in several contexts to capture the long-time behaviour of stochastic 
(partial) differential equations, see for instance~\cite{CrDeFl,GeLiRo,Arn,Cara,Deb} and the 
references therein. In the stochastic case pullback attractors are random invariant compact 
sets of phase space that are invariant with respect to the dynamics. They can be viewed as the 
generalization of non-autonomous attractors for deterministic systems. In the context of coupled 
SPDEs and SODEs, to our best knowledge, only random attractors for the stochastic FitzHugh-Nagumo 
equation were treated under various assumptions of the reaction and noise terms: finite-dimensional 
additive noise on bounded and unbounded domains~\cite{Wang, Wang2} and for (non-autonomous) 
FitzHugh-Nagumo equation driven by linear multiplicative noise~\cite{Adi1, Li, Zhou}. Here we 
provide a general random attractor theory for stochastic partly dissipative systems perturbed by 
infinite-dimensional additive noise, which goes beyond the FitzHugh-Nagumo system. To this aim we 
have to employ more general techniques than those used in the references specified above. Furthermore, we emphasize that  other dynamical aspects for similar systems have been investigated, e.g. inertial manifolds and master-slave synchronization in reference \cite{Chuesov}.  

We also mention that numerous extensions of our work are imaginable. Evidently the fully dissipative
case is easier from the viewpoint of attractor theory. Hence, our results can be extended in a 
straightforward way to the case when both components of the SPDE  contain a Laplacian. Systems with
more than two components but with similar assumptions are likely just going to
pose notational problems rather than intrinsic ones. From the point of view of applications it would be 
meaningful to incorporate non-linear couplings between the PDE and ODE parts. For example, this would allow 
us to use this theory to analyse various systems derived in chemical kinetics from mass-action laws. However, 
more complicated non-linear couplings are likely to be far more challenging. Moreover, 
one could also develop a general framework which allows one to deal with other random influences, 
e.g.~multiplicative noise, or more general Gaussian processes than standard trace-class Wiener processes. 
Furthermore, it would be interesting to investigate several dynamical aspects of partly dissipative 
SPDEs such as invariant manifolds or patterns. Naturally, one could also aim to derive upper bounds for 
the Hausdorff dimension of the random attractor and compare them to the deterministic result 
given in~\cite{Mar}.

This paper is structured as follows: Section \ref{sec:stochastics} contains all the preliminaries. 
More precisely, in Section \ref{subsec:basics} we define the system that we are going to analyse 
and state all the required assumptions. Subsequently, in Section \ref{subsec:mild}, we clarify the 
notion of solution that we are interested in. The main contribution of this work is given in 
Section \ref{sec:random}. Firstly, some preliminary definitions and results about random attractor 
theory are summarized in Section \ref{subsec:pre}. Secondly, we derive the random dynamical system 
associated to our SPDE system in Section \ref{subsec:rds}. Thirdly, we prove 
the existence of a bounded absorbing set for the random dynamical system in Section \ref{subsec:bound}. 
Lastly, in Section \ref{subsec:compact} it is shown that one can indeed find a compact absorbing set 
implying the existence of a random attractor. In Section~\ref{sec:applications} we illustrate the 
theory by several examples arising from applications.\\

\paragraph{Notation:} Before we start, we define/recall some standard notations that we will use within 
this work. When working with vectors we use $(\cdot)^\top$ to denote the transpose while $|\cdot|$ 
denotes the Euclidean norm. In a metric space $(M,d)$ we denote a ball of radius $r>0$ centred in 
the origin by $$B(r)=\{x\in M|d(x,0)\leq r\}.$$ 
We write $\text{Id}$ for the identity operator/matrix. $L(U,H)$ denotes the space of bounded linear 
operators from $U$ to $H$. $O^*$ denotes the adjoint operator of a bounded linear 
operator $O$. We let $D\subset \mathbb R^n$ always be bounded, open, and with regular boundary, where 
$n\in \mathbb N$. $L^p(D)$, $p\geq 1$, denotes the usual Lebesgue space with norm $\|\cdot\|_p$. 
Furthermore, $\langle\cdot,\cdot \rangle$ denotes the associated scalar-product in $L^2(D)$. 
$C^p(D)$, $p\in \mathbb N\cup\{0,\infty\}$, denotes the space of all continuous functions that 
have continuous first $p$ derivatives. Lastly, for  $k\in \mathbb N$, $1\leq p\leq\infty$ we consider 
the Sobolev space of order $k$ as 
$$\displaystyle W^{k,p}(D)=\left\{u\in L^{p}(D ):D^{\alpha }u\in L^{p}(D)\,\,\forall 
|\alpha |\leqslant k\right\},$$
with multi-index $\alpha$, where the norm is given by  
$$ \displaystyle \|u\|_{W^{k,p}(D )}:={\begin{cases}\left(\sum _{|\alpha |\leqslant k}
\left\|D^{\alpha }u\right\|_{L^{p}(D)}^{p}\right)^{\frac {1}{p}}&1\leqslant p<\infty ;
\\\max _{|\alpha |\leqslant k}\left\|D^{\alpha }u\right\|_{L^{\infty }(D )}&p=\infty .\end{cases}}$$ 
The Sobolev space $W^{k,p}(D)$ is a Banach space. $H_0^k(D)$ denotes the space of functions 
in $H^k(D)=W^{k,2}(D)$ that vanish at the boundary (in the sense of traces).

\section{Stochastic partly dissipative systems}
\label{sec:stochastics}

\subsection{Basics}
\label{subsec:basics}
Let $D\subset\mathbb R^n$ be a bounded open set with regular boundary, set $H:=L^2(D)$ and let $U_1,U_2$ 
be two separable Hilbert spaces. We consider the following coupled, partly dissipative 
system with additive noise
\begin{align}
&\txtd u_1=(d\Delta u_1-h(x,u_1)-f(x,u_1,u_2))~\txtd t+B_1~\txtd W_1,\label{eqn:PDE}\\
&\txtd u_2=(-\sigma(x)u_2-g(x,u_1)) ~\txtd t+B_2~\txtd W_2,
\label{eqn:ODE}
\end{align}
where $u_{1,2}=u_{1,2}(x,t)$, $(x,t)\in D\times[0,T]$, $T>0$, $W_{1,2}$ are cylindrical 
Wiener processes on $U_1$ respectively $U_2$, and $\Delta$ is the Laplace operator. 
Furthermore, $B_1\in L(U_1,H)$, $B_2\in L(U_2,H)$ and $d>0$ is a parameter 
controlling the strength of the diffusion in the first component. The system is 
equipped with initial conditions 
\begin{equation}
\label{eqn:inicond}
u_1(x,0)=u_1^0(x), ~~~ u_2(x,0)=u_2^0(x),
\end{equation} 
and a Dirichlet boundary condition for the first component
\begin{equation}
\label{eqn:boundcond}
u_1(x,t)=0 ~~~\text{on }\partial D\times [0,T].
\end{equation}
We will denote by $A$ the realization of the Laplace operator with Dirichlet 
boundary conditions, more precisely we define the operator $A:\cD(A)\rightarrow L^2(D)$ 
as $Au=d\Delta u$ with domain $\cD(A):=H^2(D)\cap H_0^1(D)\subset L^2(D)$. Note that 
$A$ is a self-adjoint operator that possesses a complete orthonormal system of 
eigenfunctions $\{e_k\}_{k=1}^\infty$ of $L^2(D)$. Within this work we always assume 
that there exists $\kappa>0$ such that $|e_k(x)|^2<\kappa$ for $k\in \mathbb N$ and $x\in D$. 
This holds for example when $D=[0,\pi]^n$. For the deterministic reaction terms 
appearing in (\ref{eqn:PDE})-(\ref{eqn:ODE}) we assume that:

\begin{ass}\label{ass:1}(Reaction terms)
\begin{enumerate}[label=(\arabic*),ref=(\arabic*)]
\item\label{ass:h} $h\in C^2(\mathbb R^n\times \mathbb R)$ and there exist $\delta_1,\delta_2,
\delta_3>0$, $p>2$ such that 
\begin{equation}
\label{eqn:condh}
\delta_1|u_1|^p-\delta_3\leq h(x,u_1)u_1\leq \delta_2|u_1|^p+\delta_3.
\end{equation}
\item\label{ass:f} $f\in C^2(\mathbb R^n\times \mathbb R\times \mathbb R)$ and there exist $\delta_4>0$ and 
$0<p_1<p-1$ such that 
\begin{equation}
\label{eqn:condf}
|f(x,u_1,u_2)|\leq \delta_4 (1+|u_1|^{p_1}+|u_2|).
\end{equation}
\item\label{ass:sigma} $\sigma\in C^2(\mathbb R^n)$ and there exist 
$\delta,\tilde \delta>0$ such that 
\begin{equation}
\label{eqn:condsi}
\delta\leq \sigma(x)\leq \tilde \delta.
\end{equation}
\item\label{ass:g} $g\in C^2(\mathbb R^n\times \mathbb R)$ and there exists $\delta_5>0$ such that 
\begin{equation}\label{eqn:condg}
|g_u(x,u_1)|\leq \delta_5,~~ |g_{x_i}(x,u_1)|\leq \delta_5(1+|u_1|),~~~i=1,\ldots,n.
\end{equation}
\end{enumerate}
\end{ass}
In particular, Assumptions \ref{ass:1} \ref{ass:h} and \ref{ass:g} imply that there exist 
$\delta_7,\delta_8>0$ such that
\begin{align}
|g(x,\xi)|&\leq \delta_7(1+|\xi|),~~~~~~~ \text{for all } 
\xi \in \mathbb R, ~x\in D,\label{eqn:condgnew}\\
|h(x,\xi)|&\leq \delta_8(1+|\xi|^{p-1}),~~~~\text{for all }
\xi \in \mathbb R, ~x\in D.\label{eqn:condhnew}
\end{align}
The Assumptions \ref{ass:1}\ref{ass:h}-\ref{ass:g} are identical to those given in~\cite{Mar}, 
except that in the deterministic case only a lower bound on $\sigma$ was assumed.

We always consider an underlying filtered probability space denoted as
$(\Omega,\mathcal F,(\mathcal F_t)_{t\geq 0},\mathbb P)$ that will be specified later on. 
In order to guarantee certain regularity properties of the noise terms, we make the following 
additional assumptions:
\begin{ass}\label{ass:2}(Noise)
\begin{enumerate}[label=(\arabic*),ref=(\arabic*)]
\item\label{ass:B2} We assume that $B_2:U_2\rightarrow H$ is a Hilbert-Schmidt operator. 
In particular, this implies that $Q_2:=B_2B_2^*$ is a trace class operator and $B_2W_2$ 
is a $Q_2$-Wiener process. 
\item\label{ass:B1} We assume that $B_1\in L(U_1,H)$ and that the operator $Q_t$ defined 
by 
\benn
Q_tu=\int_0^t\e{sA}Q_1\e{sA^*}u~\txtd s, ~~~u\in H, t\geq 0,
\eenn
where $Q_1:=B_1B_1^*$, is of trace class. Hence, $B_1W_1$ is a  $Q_1$-Wiener process as well. 
\item\label{ass:basis} Let $U_1=H$. There exists an orthonormal basis $\{e_k\}_{k=1}^\I$ of $H$ 
and sequences $\{\lambda_k\}_{k=1}^\I$ and $\{\delta_k\}_{k=1}^\I$ such that 
\benn
A e_k=-\lambda_ke_k,\qquad Q_1e_k=\delta_ke_k,~~k\in \mathbb N.
\eenn
Furthermore, we assume that there exists $\alpha\in\left(0,\frac{1}{2}\right)$ such that 
\benn
\sum_{k=1}^\infty \delta_k\lambda_k^{2\alpha+1}<\infty.
\eenn
\end{enumerate}
\end{ass}
Assumptions \ref{ass:2} guarantee that the stochastic convolution introduced 
below is a well-defined process with sufficient regularity properties, see Lemma \ref{lem:temp} and Lemma  \ref{lem:firstcomp}. 
As an example, one could choose $B_1=(-A)^{-\gamma/2}$ with $\gamma>\frac{n}{2}-1$ 
to ensure that Assumptions \ref{ass:2} \ref{ass:B1}-\ref{ass:basis} hold for $\alpha$ with $2\alpha <
\gamma-\frac{n}{2}+1$, see \cite[Chapter 4]{Prato}.

Let us now formulate problem (\ref{eqn:PDE})-(\ref{eqn:ODE}) as an abstract Cauchy problem. 
We define the following space
\begin{equation*}
\mathbb H:=L^2(D)\times L^2(D),
\end{equation*}
with norm $\|(u_1,u_2)^\top\|_{\mathbb H}^2=\|u_1\|_{2}^2+\|u_2\|_{2}^2$ this becomes 
a separable Hilbert space. $\langle\cdot,\cdot\rangle_\mathbb H$ denotes the corresponding 
scalar product. Furthermore, we let
\begin{equation*}
\mathbb V:=H_0^1(D)\times L^2(D),
\end{equation*}
with norm $\|(u_1,u_2)^\top\|_\mathbb V^2=\|u_1\|_{H^1(D)}^2+\|u_2\|_{2}^2$. We define the 
following linear operator 
\begin{equation*}
\mathbf A:=\begin{pmatrix}A&0\\0&-\sigma(x)\end{pmatrix},
\end{equation*} 
where $\mathbf A:\cD(\mathbf A)\subset \mathbb H\rightarrow \mathbb H$ with $\cD(\mathbf A)
=\cD(A)\times L^2(D)$. Since all the reaction terms are twice continuously differentiable 
they obey in particular the Carath\'eodory conditions~\cite{Zeid2B}. Thus, the corresponding 
Nemytskii operator is defined by 
\begin{align*}
\mathbf F((u_1,u_2)^\top)(x)&:=\begin{pmatrix} F_1((u_1,u_2)^\top)(x)\\F_2((u_1,u_2)^\top)(x)\end{pmatrix},\\
&:=\begin{pmatrix}-h(x,u_1(x))-f(x,u_1(x),u_2(x))\\ -g(x,u_1(x))\end{pmatrix},
\end{align*}
where $\mathbf F:\cD(\mathbf F)\subset \mathbb H\rightarrow \mathbb H$ and 
$\cD(\mathbf F):=\mathbb H$. By setting 
\benn
\mathbf W:=\begin{pmatrix}W_1\\W_2\end{pmatrix}, \qquad 
\mathbf B:=\begin{pmatrix}B_1\\B_2\end{pmatrix},\quad \text{and}\quad 
u:=\begin{pmatrix}u_1\\u_2\end{pmatrix} 
\eenn
we can rewrite the system (\ref{eqn:PDE})-(\ref{eqn:ODE}) as an abstract Cauchy 
problem on the space $\mathbb H$
\begin{equation}
\label{eqn:abstract}
\txtd u=(\mathbf A u+\mathbf F(u))~\txtd t+\mathbf B ~\txtd \mathbf W,
\end{equation}
with initial condition 
\begin{equation}
\label{eqn:abstractini}
u(0)=u^0:=\begin{pmatrix}u_1^0\\u_2^0\end{pmatrix}. 
\end{equation} 

\subsection{Mild solutions and stochastic convolution}
\label{subsec:mild}
We are interested in the concept of mild solutions to SPDEs. First of all, let us 
note the following. We have 
\begin{equation*}
\mathbf A= \underbrace{\begin{pmatrix}A&0\\0&0\end{pmatrix}}_{=:A_1}+
\underbrace{\begin{pmatrix}0&0\\0&-\sigma(x)\end{pmatrix}}_{=:A_2}.
\end{equation*}
It is well known that $A_1$ generates an analytic semigroup on $\mathbb H$ and 
$A_2$ is a bounded multiplication operator on $\mathbb H$. Hence, $\mathbf A$ is 
the generator of an analytic semigroup $\{\e{t\mathbf A}\}_{t\geq 0}$ on $\mathbb H$ 
as well, see \cite[Chapter 3, Theorem 2.1]{Pazy}. Also note that $A$ generates an 
analytic semigroup $\{\e{tA}\}_{t\geq 0}$ on $L^p(D)$ for every $p\geq 1$. In particular, 
we have for $u\in L^p(D)$ that for every $\alpha\geq 0$ there exists a constant 
$C_\alpha>0$ such that 
\begin{equation*}
\|(-A)^\alpha \e{tA}u\|_p\leq C_\alpha t^{-\alpha}\e{a t}\|u\|_p, ~ ~ ~ \text{ for all }t>0,
\end{equation*}
where $a>0$, see for instance \cite[Theorem 37.5]{SellYou}. The domain $\cD((-A)^\alpha)$ 
can be identified with the Sobolev space $W^{2\alpha,p}(D)$ and thus we have in our setting 
for $t>0$
\begin{equation}\label{eqn:ansem}
\|\e{tA}u\|_{W^{\alpha,p}(D)}\leq C_\alpha t^{-\alpha/2}\e{a t} \|u\|_p.
\end{equation} 

\begin{rem}
\label{rem:Mariongit}
Omitting the additive noise term in equation \eqref{eqn:abstract}, we are in the 
deterministic setting of \cite{Mar}. From there the existence of a global-in-time 
solution $(u_1,u_2)\in C([0,\infty),\mathbb H)$ for every initial condition 
$u^0\in \mathbb H$ already follows. 
\end{rem}

Let us now return to the stochastic Cauchy problem \eqref{eqn:abstract}-\eqref{eqn:abstractini}. 
We define

\begin{defn}(Stochastic convolution) The stochastic process defined as 
\begin{equation*}
W_\mathbf A(t):=\begin{pmatrix}W_\mathbf A^1(t)\\W_\mathbf A^2(t)\end{pmatrix}
:=\int_0^t\e{(t-s)\mathbf A}\mathbf B
~\txtd\mathbf W(s),
\end{equation*}
is called stochastic convolution. 
\end{defn} 

More precisely, we have (see \cite[Proposition 3.1]{Nagel})
\begin{align*}
W_\mathbf A(t)&=\int_0^t\begin{pmatrix}\e{(t-s)A}&0\\0&\e{-(t-s)\sigma(x)}\end{pmatrix}
\begin{pmatrix}B_1\\B_2\end{pmatrix}~\txtd\mathbf W(s)\\
&=\begin{pmatrix}\int_0^t\e{(t-s)A}B_1~\txtd W_1(s)\\\int_0^t\e{-(t-s)\sigma(x)}
B_2~\txtd W_2(s)\end{pmatrix}.
\end{align*}
This is a well-defined $\mathbb H$-valued Gaussian process. Furthermore, 
Assumptions \ref{ass:2} (1) and (2) ensure that $W_\mathbf A(t)$ is mean-square 
continuous and $\mathcal F_t$-measurable, see \cite{PratoZab}. 

\begin{rem}
\label{rem:moments}
As $W_{\mathbf A}$ is a Gaussian process, we can bound all its higher-order 
moments, i.e.~for $p\geq 1$ we have 
\begin{equation}
\sup_{t\in[0,T]}\mathbb E\|W_\mathbf A(t)\|_\mathbb H^p<\infty.
\end{equation}
This follows from the Kahane-Khintchine inequality, see~\cite[Theorem 3.12]{Van}.
\end{rem}

\begin{defn}(Mild solution) A mean-square continuous, $\mathcal F_t$-measurable 
$\mathbb H$-valued process $u(t)$, $t\in [0,T]$ is said to be a mild solution 
to (\ref{eqn:abstract})-(\ref{eqn:abstractini}) on $[0,T]$
if $\mathbb P$-almost surely we have for $t\in [0,T]$
\begin{equation}
\label{eqn:varconst}
u(t)=\e{t\mathbf A}u^0+\int_0^t\e{(t-s)\mathbf A}\mathbf F(u(s))
~\txtd s+W_\mathbf A(t).
\end{equation}
\end{defn}

Under Assumptions \ref{ass:1} and \ref{ass:2} (1)-(2) a mild solution exists 
locally-in-time in 
$$L^2(\Omega;C([0,T];\mathbb H))\cap L^2(\Omega;L^2([0,T];\mathbb V)),$$ 
for some $T>0$, see \cite[Theorem 7.7]{PratoZab}. Hence, local in time existence 
for our problem is guaranteed by the classical SPDE theory. 

\section{Random attractor}
\label{sec:random}

\subsection{Preliminaries}
\label{subsec:pre}
We now recall some basic definitions related to random attractors. For more information the reader 
is referred to the sources given in the introduction.

\begin{defn}(Metric dynamical system) Let $(\Omega, \mathcal F,\mathbb P)$ be a probability space
and let $\theta=\{\theta_t:\Omega\rightarrow \Omega\}_{t\in\R}$ be a family of $\mathbb P$-preserving 
transformations (i.e. $\theta_t\mathbb P=\mathbb P$ for $t\in \mathbb R$), which satisfy for $t,s\in\R$ that 
\begin{itemize}
\item[(1)] $(t,\omega)\mapsto \theta_t\omega$ is measurable,
\item[(2)] $\theta_0=\text{Id}$,
\item[(3)] $\theta_{t+s}=\theta_t\circ\theta_s$. 
\end{itemize} 
Then $(\Omega,\mathcal F, \mathbb P,\theta)$ is called a metric 
dynamical system. 
\end{defn}
The metric dynamical system describes the dynamics of the noise. 
\begin{defn}(Random dynamical system) Let $(\cV,\|\cdot\|)$ be a separable Banach space. 
A random dynamical system (RDS) with time domain $\R_+$ on $(\cV,\|\cdot\|)$ over $\theta$ is 
a measurable map 
\benn
\varphi:\R_+\times \cV\times \Omega\rightarrow \cV; ~~~ (t,v,\omega)\mapsto \varphi(t,\omega)v
\eenn
such that $\varphi(0,\omega)=\text{Id}_{\cV}$ and 
\benn
\varphi(t+s,\omega)=\varphi(t,\theta_s\omega)\circ\varphi(s,\omega)
\eenn
for all $s,t\in \R_+$ and for all $\omega \in \Omega$. We say that $\varphi$ is a continuous or 
differentiable RDS if $v\mapsto \varphi(t,\omega)v$ is continuous or differentiable for all 
$t\in \R_+$ and every $\omega\in \Omega$. 
\end{defn}
We summarize some further definitions relevant for the theory of random attractors. 

\begin{defn}(Random set) A set-valued map $\cK:\Omega\rightarrow 2^\cV$ is said to be measurable 
if for all $v\in \cV$ the map $\omega\mapsto d(v,\cK(\omega))$ is measurable. Here, $d(\cA,\cB)
=\sup_{v\in \cA}\inf_{\tilde{v}\in \cB}\|v-\tilde{v}\|$ for $\cA,\cB\in 2^\cV$, $\cA,\cB\neq 
\emptyset$ and $d(v,\cB)=d(\{v\},\cB)$. A measurable set-valued map is called a random set. 
\end{defn}

\begin{defn}(Omega-limit set) For a random set $\cK$ we define the omega-limit set to be 
\benn
\Omega_\cK(\omega):=\bigcap_{T\geq 0}\overline{\bigcup_{t\geq T}\varphi(t,\theta_{-t}\omega)
\cK(\theta_{-t}\omega)}.
\eenn
$\Omega_\cK(\omega)$ is closed by definition.
\end{defn}

\begin{defn}(Attracting and absorbing set) Let $\cA,\cB$ be random sets and let $\varphi$ be a RDS.
\begin{itemize}
\item  $\cB$ is said to 
attract $\cA$ for the RDS $\varphi$, if
\benn 
d(\varphi(t,\theta_{-t}\omega)\cA(\theta_{-t}\omega),\cB(\omega))\rightarrow 0~~
\text{for }t\rightarrow \infty.
\eenn
\item $\cB$ is said to absorb $\cA$ for the RDS $\varphi$, if  there exists a (random) 
absorption time $t_\cA(\omega)$ such that for all $t\geq t_\cA(\omega)$ 
\benn
\varphi(t,\theta_{-t}\omega)\cA(\theta_{-t}\omega)\subset \cB(\omega).
\eenn
\item Let $\mathbf \cD$ be a collection of random sets (of non-empty subsets of $\cV$), 
which is closed with respect to set inclusion. A set $\cB\in \mathbf \cD$ is called 
$\mathbf \cD$-absorbing/$\cD$-attracting for the RDS $\varphi$, if $\cB$ absorbs/attracts 
all random sets in $\mathbf \cD$. 
\end{itemize}
\end{defn}

\begin{rem}
\label{rem:abs}
Throughout this work we use a convenient criterion to derive the existence of an 
absorbing set. Let $\cA$ be a random set. If for every $v\in \cA(\theta_{-t}\omega)$ 
we have 
\begin{equation}
\limsup_{t\rightarrow \infty} \|\varphi(t,\theta_{-t}\omega,v)\|\leq \rho(\omega),
\end{equation}  
where $\rho(\omega)>0$ for every $\omega \in \Omega$, then the ball centred in $0$ 
with radius $\rho(\omega)+\epsilon$ for a $\epsilon>0$, i.e. $\cB(\omega):=
B(0,\rho(\omega)+\epsilon)$, absorbs $\mathcal A$.
\end{rem}

\begin{defn}(Tempered set) A random set $\cA$ is called tempered provided for 
$\mathbb P$-a.e.~$\omega\in \Omega$ 
\benn
\lim_{t\rightarrow \infty} \e{-\beta t} d(\cA(\theta_{-t}\omega))=0~~~
\text{ for all }\beta>0,
\eenn
where $d(\cA)=\sup_{a\in\cA}\|a\|$.
We denote by $\mathcal T$ the set of all tempered 
subsets of $\cV$. 
\end{defn}

\begin{defn}(Tempered random variable) A random variable 
$X\in \mathbb R$ on $(\Omega,\mathcal F,\mathbb P,\theta)$ 
is called tempered, if there is a set of full $\mathbb P$-measure 
such that for all $\omega$ in this set we have
\begin{equation} \label{eqn:temrv}
\lim_{t\rightarrow \pm \infty}\frac{\log\left|X(\theta_t\omega)\right|}{|t|}=0.
\end{equation}
\end{defn}

Hence a random variable $X$ is tempered when the stationary random process 
$X(\theta_t\omega)$ grows sub-exponentially.  

\begin{rem}
\label{rem:temp}
A sufficient condition that a positive random variable $X$ is tempered is 
that (cf.~\cite[Proposition 4.1.3]{Arn})
\begin{equation}
\mathbb E\left(\sup_{t\in[0,1]}X(\theta_t\omega)\right)<\infty.
\end{equation}
\end{rem}

If $\theta$ is an ergodic shift, then the only alternative to~\eqref{eqn:temrv} 
is 
\benn \label{eqn:temrv1}
\lim_{t\rightarrow \pm \infty}\frac{\log\left|X(\theta_t\omega)\right|}{|t|}=\infty,
\eenn
i.e., the random process $X(\theta_t\omega)$ either grows sub-exponentially or 
blows up at least exponentially. 

\begin{defn}(Random attractor) Suppose $\varphi$ is a RDS such that there exists a 
random compact set $\cA\in \mathcal T$ which satisfies for any $\omega \in \Omega$
\begin{itemize}
\item $\cA$ is invariant, i.e., $\varphi(t,\omega)\cA(\omega)=\cA(\theta_t\omega)$ 
for all $t\geq0$. 
\item $\cA$ is $\mathcal T$-attracting.
\end{itemize}
Then $\cA$ is said to be a $\mathcal T$-random attractor for the RDS. 
\end{defn}

\begin{thm} (\cite{CraFla},\cite{Schmal})
\label{thm:atthm}
Let $\varphi$ be a continuous RDS and assume there exists a compact random set 
$\cB \in \mathcal T$ that absorbs every $\cD\in \mathcal T$, i.e. $\mathcal B$ is 
$\mathcal T$-absorbing. Then there exists a unique $\mathcal T$-random attractor 
$\cA$, which is given by 
\begin{equation*}
\mathcal{A}(\omega)=\overline{\bigcup_{\mathcal D \in \cT} \Omega_\cD(\omega)}.
\end{equation*}
\end{thm}

We will use the above theorem to show the existence of a random attractor for the 
partly dissipative system at hand. 

\subsection{Associated RDS}
\label{subsec:rds}
We will now define the RDS corresponding to~\eqref{eqn:abstract}-\eqref{eqn:abstractini}. 
We consider $\mathcal V=\mathbb H:=L^2(D)\times L^2(D)$ and $\cT$ is the set of all tempered 
subsets of $\mathbb H$.  In the sequel, we consider the fixed canonical probability space 
$(\Omega,\mathcal F,\mathbb P)$ corresponding to a two-sided Wiener process, more precisely
\begin{align*}
\Omega:=&\left\{\omega=(\omega_1,\omega_2): \omega_1,\omega_2 \in C(\mathbb R,L^2(D)), 
\omega(0)=0\right\},
\end{align*}
endowed with the compact-open topology.
The $\sigma$-algebra $\mathcal F$ is the Borel $\sigma$-algebra on 
$\Omega$ and $\mathbb P$ is the distribution of the trace class Wiener process 
$\tilde{W}(t):=(\tilde W_1(t),\tilde W_2(t))=(B_1W_1(t),B_2W_2(t))$, where we recall that $B_1$ and $B_2$ fulfil Assumptions \ref{ass:2}. We identify the 
elements of $\Omega$ with the paths of these Wiener processes, more precisely
\begin{equation}
\tilde W(t,\omega):=(\tilde W_1(t,\omega_{1}),\tilde W_2(t,\omega_{2}))=
(\omega_1(t),\omega_2(t))=:\omega(t), \mbox{ for } \omega\in\Omega.
\end{equation}
Furthermore, we introduce the Wiener shift, namely
\begin{equation}
\theta_t\omega(\cdot):=\omega(\cdot+t)-\omega(t), ~\mbox{for }\omega\in 
\Omega \mbox{ and }
t\in \mathbb R.
\end{equation}
Then $\theta:\mathbb R\times \Omega \rightarrow \Omega$ is a measure-preserving 
transformation on $\Omega$, i.e.~$\theta_{t}\mathbb{P}=\mathbb{P}$, for $t\in\mathbb{R}$. 
Furthermore, $\theta_0\omega(s)=\omega(s)-\omega(0)=\omega(s)$ and $\theta_{t+s}\omega(r)
=\omega(r+t+s)-\omega(t+s)=\theta_t(\omega(r+s)-\omega(s))=\theta_t(\theta_s\omega(r))$. 
Hence, $(\Omega,\mathcal F,\mathbb P,\theta)$ is a metric 
dynamical system. Next, we consider the following equations 
\begin{align}\label{eqn:eqnoi1}
\txtd z_1&=Az_1~\txtd t+~\txtd \omega_1,\\
\txtd z_2&= -\sigma(x)z_2~\txtd t + \txtd \omega_2. \label{eqn:eqnoi2}
\end{align}
The stationary solutions of \eqref{eqn:eqnoi1}-\eqref{eqn:eqnoi2} are given by
\begin{align*}
(t,\omega)\mapsto z_{1}(\theta_{t}\omega) \mbox { and } 
(t,\omega)\mapsto z_{2}(\theta_t\omega),
\end{align*}
where 
\begin{align*}
z_1(\theta_{t}\omega)&=\int\limits_{-\infty}^{t}e^{(t-s)A }~\txtd \omega_1(s)
=\int\limits_{-\infty}^{0}e^{-s A}~\txtd \theta_{t}\omega_{1}(s),
\\
z_2(\theta_{t}\omega)& = \int\limits_{-\infty}^{t} e^{-(t-s)\sigma(x)}~
\txtd \omega_{2}(s) =\int\limits_{-\infty}^{0} e^{s \sigma(x)}~\txtd \theta_{t}\omega_{2}(s).
\end{align*}
Here, we observe that for $t=0$ 
\begin{align*}
z_{1}(\omega)=\int\limits_{-\infty}^{0} e^{-s A}~\txtd \omega_{1}(s), ~ ~ ~
z_{2}(\omega) =\int\limits_{-\infty}^{0} e^{s\sigma(x)}~\txtd \omega_{2}(s).
\end{align*}
Now consider the Doss-Sussmann transformation $v(t)=u(t)-z(\theta_t \omega)$, 
where $v(t)=(v_1(t),v_2(t))^\top$, $z(\omega)=(z_1(\omega_{1}),z_2(\omega_{2}))^\top$ and 
$u(t)=(u_1(t),u_2(t))^\top$ is a solution to the problem 
(\ref{eqn:PDE})-(\ref{eqn:boundcond}). Then $v(t)$ satisfies 
\begin{align}
\label{eqn:eqv}
\frac{\txtd v}{\txtd t}&=\mathbf A v+\mathbf F(v+z(\theta_t\omega)).
\end{align}
More explicitly / or component-wise this reads as
\begin{align}
\frac{\txtd v_1(t)}{\txtd t}&=d\Delta v_1(t)-h(x,v_1(t)+z_1(\theta_t\omega))\nonumber\\
&\qquad
-f(x,v_1(t)+z_1(\theta_t\omega),v_2(t)+z_2(\theta_t\omega)),\label{eqn:RDS1}\\
\frac{\txtd v_2(t)}{\txtd t}&=-\sigma(x)v_2(t)-g(x,v_1(t)+z_1(\theta_t\omega))
\label{eqn:RDS2}.
\end{align}
In the equations above no stochastic differentials appear, hence 
they can be considered path-wise, i.e., for every $\omega$ instead just 
for $\mathbb P$-almost every $\omega$. For every $\omega$ \eqref{eqn:eqv} is a 
deterministic equation, where $z(\theta_t\omega)$ can be regarded as a 
time-continuous perturbation. In particular, \cite{ChepVish} guarantees that 
for all $v^0=(v_1^0,v_2^0)^\top\in \mathbb H$ there 
exists a solution $v(\cdot,\omega,v^0)\in C([0,\infty),\mathbb H)$  with  
$v_1(0,\omega,v_1^0)=v_1^0$, $v_2(0,\omega,v_2^0)=v_2^0$. Moreover, the mapping 
$\mathbb{H} \ni v_{0}\mapsto v(t,\omega,v_{0})\in\mathbb{H}$ is continuous. 
Now, let
\begin{align*}
u_1(t,\omega,u_1^0)&=v_1(t,\omega,u_1^0-z_1(\omega))+z_1(\theta_t\omega), \\
u_2(t,\omega,u_2^0)&=v_2(t,\omega,u_2^0-z_2(\omega))+z_2(\theta_t\omega).
\end{align*}
Then $u(t,\omega,u^0)=(u_1(t,\omega,u_1^0),u_2(t,\omega,u_2^0))^\top$ is a 
solution to (\ref{eqn:PDE})-(\ref{eqn:boundcond}). In particular, we can conclude at this point that (\ref{eqn:PDE})-(\ref{eqn:boundcond}) has a global-in-time 
solution which belongs to $C([0,\infty);\mathbb{H})$; see Remark~\ref{rem:Mariongit}.
We define the corresponding solution operator $\varphi:\mathbb R^+\times 
\Omega\times \mathbb H\rightarrow \mathbb H$ as
\begin{equation}
\label{eqn:ourRDS}
\varphi(t,\omega,(u_1^0,u_2^0)):=(u_1(t,\omega,u_1^0),u_2(t,\omega,u_2^0)),  
\end{equation}
for all $(t,\omega,(u_1^0,u_2^0))\in \mathbb R^{+}\times 
\Omega\times \mathbb H$. Now, $\varphi$ is a continuous RDS associated to our stochastic partly dissipative 
system. In particular, the cocycle property obviously follows from the mild formulation. 
In the following, we will prove the existence of a global random attractor for this 
RDS. Due to conjugacy, see~\cite{CraFla, Schmal} this gives us automatically a global 
random attractor for the stochastic partly dissipative 
system \eqref{eqn:PDE}-(\ref{eqn:boundcond}). 

\subsection{Bounded absorbing set}
\label{subsec:bound}

In the following we will prove the existence of a bounded absorbing set for the 
RDS \eqref{eqn:ourRDS}. In the calculations we will make use of some versions 
of certain classical deterministic results several times. Therefore, we recall 
these results here for completeness and as an aid to follow the calculations later
on.
 
\begin{lem}($\varepsilon$-Young inequality) 
For $x,y\in\mathbb R$, $\varepsilon>0$, $\tilde p, \tilde q>1$, 
$\frac{1}{\tilde p}+\frac{1}{\tilde q}=1$ we have
\begin{equation}
|xy|\leq \varepsilon |x|^{\tilde p}+\frac{(\tilde p 
\varepsilon)^{1-\tilde q}}{\tilde q}|y|^{\tilde q}.
\end{equation}
\end{lem}

\begin{lem}(Gronwall's inequality)
Assume that $\varphi$, $\alpha$ and $\beta$ are integrable functions and 
$\varphi(t)\geq 0$. If 
\begin{equation}
\varphi'(t)\leq \alpha(t)+\beta(t)\varphi(t),
\end{equation}
then 
\begin{equation}
\varphi(t)\leq\varphi(t_0)\e{\int_{t_0}^t\beta(\tau)d\tau}+
\int_{t_0}^t\alpha(s)\e{\int_s^t\beta(\tau)d\tau}ds, ~ ~ ~ t\geq t_0.
\end{equation}
\end{lem}

\begin{lem}(Uniform Gronwall Lemma~\cite[Lemma 1.1]{Temam}) 
Let $g$, $h$, $y$  be positive locally integrable functions on 
$(t_0,\infty)$ such that $y'$ is locally integrable on $(t_0,\infty)$ 
and which satisfy 
$$\frac{\txtd y}{\txtd t}\leq gy+h, ~~~~ \text{ for }t\geq t_0,$$
$$\int_t^{t+r}g(s)\txtd s\leq  a_1,~~~\int_t^{t+r}h(s)\txtd s\leq a_2,
~~~ \int_t^{t+r}y(s)\txtd s\leq a_3~~~\text{ for }t\geq t_0,$$
where $r,a_1,a_2,a_3$ are positive constants. Then 
$$y(t+r)\leq \left(\frac{a_3}{r}+a_2\right) \e{a_1},~~~~\forall t\geq t_0.$$
\end{lem}

\begin{lem}(Minkowski's inequality) 
Let $p>1$ and $f,g\in \mathbb R$, then 
$$|f+g|^p\leq 2^{p-1}(|f|^p+|g|^p).$$
\end{lem}

\begin{lem}(Poincar\'e's inequality) 
Let $1\leq p < \infty$ and let $D\subset\mathbb R^n$ be a 
bounded open subset. Then there exists a constant $c= c(D,p)$ 
such that for every function $u\in W_0^{1,p}(D)$ 
\begin{equation} \|u\|_{p}\leq c\|\nabla u\|_{p}.
\end{equation} 
\end{lem}

Having recalled the relevant deterministic preliminaries, we can
now proceed with the main line of our argument. For the following 
result about the stochastic convolutions Assumption \ref{ass:2} (3) 
is crucial. 

\begin{lem}\label{lem:temp}
Suppose Assumptions \ref{ass:1} and \ref{ass:2} hold. 
Then for every $p\geq 1$ 
$$\|z_1(\omega)\|_p^p \text{ and }\|z_2(\omega)\|_2^2$$ are tempered 
random variables. 
\end{lem}

\begin{proof}
Using $0<\delta\leq\sigma(x)\leq \tilde \delta$ and the Burkholder-Davis-Gundy 
inequality we have 
\begin{align*}
&\mathbb E\left(\sup_{t\in [0,1]}\|z_2(\theta_t\omega)\|_2^2\right)\\&
=\mathbb E\left(\sup_{t\in [0,1]}\left\|\int_{-\infty}^t\e{-(t-s)\sigma(x)}~
\txtd \omega_2(s)\right \|_2^2\right)\\
&=\mathbb E\left(\sup_{t\in[0,1]}\int_D\e{-2t\sigma(x)}
\left|\int_{-\infty}^t\e{s\sigma(x)}~\txtd \omega_2(s)\right|^2~\txtd x\right)\\
&\leq \mathbb E\left(\sup_{t\in[0,1]}\e{-2t\delta}\int_D\left|
\int_{-\infty}^t\e{s\sigma(x)}~\txtd \omega_2(s)\right|^2~\txtd x\right)\\
&\leq \mathbb E\left(\sup_{t\in[0,1]}\left\|\int_{-\infty}^t\e{s\sigma(x)}~
\txtd \omega_2(s)\right\|_2^2\right)\\
&\leq C\mathbb E\left(\int_{-\infty}^1\|\e{s\sigma(x)}\|_2^2~\txtd s\right)\\
&\leq C|D|\int_{-\infty}^1\e{2s\tilde \delta}~\txtd s= \frac{C|D|}{2\tilde \delta} 
\e{2\tilde \delta}\\
&<\infty.
\end{align*}
The temperedness of $\|z_2(\omega)\|_2^2$ then follows directly using Remark~\ref{rem:temp}. 
Now, we consider the random variable $\|z_1(\omega)\|_p^p$.  Note that using the 
so-called factorization method we have for $(x,t)\in D\times [0,T]$ and $\alpha\in (0,1/2)$ 
(see \cite[Ch. 5.3]{PratoZab})
\begin{equation}
z_1(x,\theta_t\omega)=\frac{\sin(\pi \alpha)}{\pi}\int_{-\infty}^t\e{(t-\tau)A}(t-\tau)^{\alpha-1} 
Y(x,\tau)~\txtd \tau,
\end{equation}
with 
\begin{align*}
Y(x,\tau)&=\int_0^\tau\e{(\tau-s)A}(\tau-s)^{-\alpha} B_1~\txtd W_1(x,s)\\
&=\sum_{k=1}^\infty \int_0^\tau \e{(\tau-s)A}(\tau-s)^{-\alpha}B_1e_k(x)d\beta_k(s)\\
&=\sum_{k=1}^\infty \int_0^\tau \e{-(\tau-s)\lambda_k}(\tau-s)^{-\alpha}\sqrt{\delta_k}
e_k(x)d\beta_k(s),
\end{align*}
where we have used the formal representation $W_1(x,s)=\sum_{k=1}^\infty \beta_k(s)e_k(x)$ 
of the cylindrical Wiener process, with $\{\beta_k\}_{k=1}^\infty$ being a sequence of 
mutually independent real-valued Brownian motions. 
$Y(x,\tau)$ is a real-valued Gaussian random variable with mean zero and variance 
\begin{align*}
&\text{Var}(Y(x,\tau))=\mathbb E\left[|Y(x,\tau)|^2\right]\\
&=\mathbb E\left[\sum_{k=1}^\infty\left( \int_0^\tau \e{-(\tau-s)\lambda_k}(t-s)^{-\alpha}
\sqrt{\delta_k}~\txtd \beta_k(s)\right)^2|e_k(x)|^2\right]\\
&=\sum_{k=1}^\infty \delta_k |e_k(x)|^2 \mathbb E\left[\left( \int_0^\tau 
\e{-(\tau-s)\lambda_k}(t-s)^{-\alpha}~\txtd \beta_k(s)\right)^2\right]\\
&=\sum_{k=1}^\infty \delta_k |e_k(x)|^2  \int_0^\tau \e{-2s\lambda_k}s^{-2\alpha}~\txtd s,
\end{align*}
where we have used Parseval's identity and the It\^o isometry. 
Our assumption on the boundedness of the eigenfunctions $\{e_k\}_{k=1}^\infty$ 
yields together with Assumption \ref{ass:2} (3) that 
\begin{align*}
\text{Var}(Y(x,\tau))&<\sum_{k=1}^\infty\delta_k \kappa^2\int_0^\infty 
\e{-2s\lambda_k}s^{-2\alpha}~\txtd s\\
&=\kappa^22^{2\alpha-1}\Gamma(1-2\alpha)\sum_{k=1}^\infty\delta_k
\lambda_k^{2\alpha-1}<\infty.
\end{align*}
Hence, $\mathbb E\left[\left|Y(x,\tau)\right|^{2m}\right] \leq C_m$ for $C_m>0$ 
and every $m\geq 1$ (note that all odd moments of a Gaussian random variable are zero). 
Thus we have 
\begin{equation*}\label{eqn:ex}
\mathbb E\left[\int_0^T\int_D|Y(x,\tau)|^{2m}\txtd x\txtd \tau\right] \leq TC_m|D|,
\end{equation*}
i.e., in particular for all $p\geq 1$ we have $Y\in L^{p}(D\times [0,T])$ $\mathbb P$-a.s.. 
We now observe 
\begin{align*}
&\|z_1(\theta_t\omega)\|_{W^{\alpha,p}(D)}\\ &\leq \frac{\sin(\pi\alpha)}{\pi}
\int_{-\infty}^t(t-\tau)^{\alpha-1}\|\e{(t-\tau)A}Y(\cdot,\tau)\|_{W^{\alpha,p}(D)}~\txtd \tau\\
&\leq C\frac{\sin(\pi\alpha)}{\pi}\int_{-\infty}^t(t-\tau)^{\alpha-1}(t-\tau)^{-{\alpha/2}}
e^{-\lambda (t-\tau)} \|Y(\cdot,\tau)\|_{p}~\txtd \tau\\
&\leq C \sup_{\tau \in (-\infty,t]}\|Y(\cdot,\tau)\|_p \int_{-\infty}^t(t-\tau)^{\alpha/2-1} 
e^{-\lambda (t-\tau)}~\txtd \tau,
\end{align*}
where we have used \eqref{eqn:ansem} and thus 
\begin{align*}
&\mathbb E\left(\sup_{t\in [0,1]}\|z_1(\theta_t\omega)\|_p\right)\\&\leq  
C~\mathbb E\left(\sup_{t\in [0,1]}\sup_{\tau \in(-\infty,t]}\|Y(\cdot, \tau)\|_p\right) 
\int_{0}^\infty\tau^{\alpha/2-1}e^{-\lambda\tau}~\txtd \tau\\
&=  C~\mathbb E\left(\sup_{t\in [0,1]}\sup_{\tau \in(-\infty,t]}\|Y(\cdot, \tau)\|_p\right) 
\frac{\Gamma(\alpha/2)}{\lambda^{\alpha/2}}.
\end{align*}
Now, the right hand side is finite as all  moments of $Y(x,\tau)$ are bounded 
uniformly in $x,\tau$, see above. Due to embedding of Lebesgue spaces on a bounded domain 
we have that
$$\mathbb E\left(\sup_{t\in [0,1]}
\|z_1(\theta_t\omega)\|_p\right)<\infty ~~ \text{ implies }~~ 
\mathbb E\left(\sup_{t\in [0,1]}\|z_1(\theta_t\omega)\|_p^p\right)<\infty,$$ 
i.e., temperedness of $\|z_1(\omega)\|_p^p$ follows again with Remark \ref{rem:temp}. 
\end{proof}

\begin{rem}\label{rem:nabla}
	\begin{itemize}
		\item [1)] Note that Assumption~\ref{ass:2} (3) together with the boundedness of $e_{k}$ for $k\in\mathbb{N}$ are essential for this proof. One can extend such statements for general open bounded domains in $D\subset\mathbb{R}^{n}$, according to Remark~5.27 and Theorem~5.28 in~\cite{PratoZab}.
	\item [2)] Regarding again Assumption~\ref{ass:2} (3) one can show in a similar 
	way that $ z_1 \in W^{1,p}(D)$ and in particular also $\|\nabla z_1(\omega)\|_p^p$ is a 
	tempered random variable for all $p\geq 1$. 
	\end{itemize}
\end{rem}

\begin{rem}
	Alternatively, one can introduce the Ornstein-Uhlenbeck processes $z_1$ and $z_2$ using integration by parts. We applied the factorization Lemma for the definition of $z_1$ in order to obtain regularity results for $z_1$ based on the interplay between the eigenvalues of the linear part and of the covariance operator of the noise. \\
	Using integration by parts, one infers that
	\begin{align*}
	z_{1}(\theta_t\omega) =\int\limits_{-\infty}^{t} \exp((t-\tau)A)~\txtd \omega_1(t)& =\omega_{1}(t) + A \int\limits_{-\infty}^{t} \exp((t-\tau)A)\omega_{1}(\tau)~\txtd \tau \\
	&= - A \int\limits_{-\infty}^{t} \exp((t-\tau)A)(\omega_{1}(t) -\omega_{1}(\tau))~\txtd \tau.
	\end{align*}
	This expression can also be used in order to investigate the regularity of $z_1$ in a Banach space $\cH$ as follows:
	\begin{align*}
	&\Big\| A \int\limits_{-\infty}^{t} \exp((t-\tau)A)(\omega_{1}(t) -\omega_{1}(\tau))~\txtd \tau\Big\|_{\cH}  \\&\leq C \int\limits_{-\infty}^{t} (t-\tau )^{-1} \|\exp(t-\tau) A\|_{\cH} \|\omega_{1}(t) -\omega_{1}(\tau)\|_{\cH}~\txtd \tau.
	\end{align*}
	Here one uses the H\"older-continuity of $\omega_{1}$ in an appropriate function space in order to compensate the singularity in the previous formula.\\
	In our case, we need $z_1\in D(A^{\alpha/2})=W^{\alpha,p}(D)$. Letting $\omega_{1}\in D(A^{\varepsilon})$ for $\varepsilon\geq 0$ and using that $\omega_1$ is $\beta$-H\"older continuous with $\beta\leq 1/2$ one has
	\begin{align*}
	\|z_{1}(\theta_{t}\omega)\|_{W^{\alpha,p}(D)} \leq \int\limits_{-\infty}^{t} (t-\tau )^{\beta+\varepsilon-\alpha/2-1} \|\omega_{1}\|_{\beta,\varepsilon} \|\exp((t-\tau)A)\|~\txtd \tau,
	\end{align*}
	which is well-defined if $\beta+\varepsilon>\alpha/2$. Such a condition provides again an interplay between the time and space regularity of the stochastic convolution.
\end{rem}

Based on the results regarding the stochastic convolutions we can now investigate the long-time behaviour of our system. The first step is contained in the next lemma, which establishes the existence of an absorbing set.

\begin{lem}\label{lem:abs}
Suppose Assumptions \ref{ass:1} and \ref{ass:2} hold. Then there exists 
a set $\{\cB(\omega)\}_{\omega\in \Omega}\in \mathcal T$ such that 
$\{\cB(\omega)\}_{\omega \in \Omega}$ is a bounded absorbing set for $\varphi$. 
In particular, for any $\cD=\{\cD(\omega)\}_{\omega\in \Omega}\in \mathcal T$ 
and every $\omega \in \Omega$ there exists a random time $t_\cD(\omega)$ such 
that for all $t\geq t_\cD(\omega)$ 
\begin{equation}
\varphi(t,\theta_{-t}\omega,\cD(\theta_{-t}\omega)) \subset \cB(\omega). 
\end{equation}
\end{lem}

\begin{proof}
To show the existence of a bounded absorbing set, we want to make use of 
Remark~\ref{rem:abs}, i.e. we need an a-priori estimate in $\mathbb H$. We have 
for $v=(v_1,v_2)^\top$ solution of \eqref{eqn:eqv}
\begin{align*}
&\frac{1}{2}\frac{\txtd}{\txtd t}\left(\|v_1\|^2_2+\|v_2\|^2_2\right)=\frac{1}{2}\frac{\txtd}{\txtd t}\|v\|_\mathbb H^2
=\left\langle \frac{\txtd}{\txtd t}v, v\right\rangle_\mathbb H =\langle \mathbf Av+\mathbf F(v+z(\theta_t\omega)),v\rangle_\mathbb H\\
&=\langle dAv_1,v_1\rangle+\langle F_1(v+z(\theta_t\omega)),v_1\rangle-\langle \sigma(x)v_2,v_2\rangle+\langle F_2(v+z(\theta_t \omega)),v_2\rangle\\
&=-d \|\nabla v_1\|_{2}^2\underbrace{-\langle h(x,v_1+z_1(\theta_t\omega)),v_1\rangle}_{=:I_1} \underbrace{-\langle f(x,v_1+z_1(\theta_t\omega),v_2+z_2(\theta_t\omega)),v_1\rangle}_{=:I_2}\\
&\qquad{}-\delta\|v_2\|_2^2 \underbrace{-\langle g(x,v_1+z_1(\theta_t\omega)),v_2\rangle}_{=:I_3},
\end{align*}
where we have used (\ref{eqn:condsi}). 
We now estimate $I_1$-$I_3$ separately. Deterministic constants denoted as $C,C_1,C_2,...$ may change from line to line. 
Using (\ref{eqn:condh}) and (\ref{eqn:condhnew}) we calculate
\begin{align*}
I_1&=-\int_Dh(x,v_1+z_1( \theta_t\omega))v_1~\txtd x\\
&=-\int_D h(x,v_1+z_1(\theta_t\omega))(v_1+z_1(\theta_t\omega))~\txtd x\\
&\qquad +\int_Dh(x,v_1+z_1(\theta_t\omega))z_1(\theta_t\omega)~\txtd x\\
&\leq -\int_D\delta_1|u_1|^p~\txtd x+\int_D\delta_3~\txtd x+\int_D|h(x,v_1+z_1(\theta_t\omega))||z_1(\theta_t\omega)|~\txtd x\\
&\leq -\delta_1\|u_1\|_{p}^p+C+\delta_8\int_D(1+|u_1|^{p-1})|z_1(\theta_t\omega)|~\txtd x\\
&=-\delta_1\|u_1\|_{p}^p+C+\delta_8\|z_1(\theta_t\omega)\|_{1}+\delta_8\int_D|u_1|^{p-1}|z_1(\theta_t\omega)|~\txtd x\\
&\leq -\delta_1\|u_1\|_{p}^p+C+C_1\|z_1(\theta_t\omega)\|_{2}^2+\frac{\delta_1}{2}\|u_1\|_{p}^p+ C_2 \|z_1(\theta_t\omega)\|_{p}^p\\
&=-\frac{\delta_1}{2}\|u_1\|_{p}^p+C+ C_1\left(\|z_1(\theta_t\omega)\|_{2}^2+ \|z_1(\theta_t\omega)\|_{p}^p\right).
\end{align*}
Furthermore, with (\ref{eqn:condf}) we estimate
\begin{align*}
I_2&=-\int_Df(x,v_1+z_1(\theta_t\omega),v_2+z_2(\theta_t\omega))v_1~\txtd x\\
&\leq \int_D|f(x,v_1+z_1(\theta_t\omega),v_2+z_2(\theta_t\omega))||u_1-z_1(\theta_t\omega)|~\txtd x\\
&\leq \int_D \delta_4(1+|u_1|^{p_1}+|u_2|)|u_1|~\txtd x\\
&\qquad+\int_D\delta_4(1+|u_1|^{p_1}+|u_2|)|z_1(\theta_t\omega)|~\txtd x\\
&=\int_D \delta_4(|u_1|+|u_1|^{p_1+1})~\txtd x+\int_D\delta_4|u_1||u_2|~\txtd x+\delta_4\|z_1(\theta_t\omega)\|_{1}\\
&\qquad +\int_D\delta_4|u_1|^{p_1}|z_1(\theta_t\omega)|~\txtd x+\int_D\delta_4|u_2||z_1(\theta_t\omega)|~\txtd x\\
&\leq \int_D \delta_4(|u_1|+|u_1|^{p_1+1})~\txtd x+\int_D\delta_4|u_1||u_2|~\txtd x+\delta_4\|z_1(\theta_t\omega)\|_{2}^2+C
\\&\qquad +\int_D \frac{\delta_4}{2}|u_1|^{p_1+1}~\txtd x+ C_1 \|z_1(\theta_t\omega)\|_{p_1+1}^{p_1+1}+\int_D\delta_4|u_2||z_1(\theta_t\omega)|~\txtd x\\ 
&\leq \int_D \delta_4\frac{3}{2}(|u_1|+|u_1|^{p_1+1})~\txtd x+\int_D\delta_4|u_1||u_2|~\txtd x+C\\
&\qquad+C_1 \left(\|z_1(\theta_t\omega)\|_{2}^2+\|z_1(\theta_t\omega)\|_{p_1+1}^{p_1+1}\right)+\int_D\delta_4|u_2||z_1(\theta_t\omega)|~\txtd x.
\end{align*}
With \eqref{eqn:condgnew} we compute
\begin{align*}
I_3&=-\int_Dg(x,v_1+z_1(\theta_t\omega))v_2~\txtd x\\
&\leq \int_D|g(x,u_1)||u_2-z_2(\theta_t\omega)|~\txtd x\\
&\leq \int_D\delta_7(1+|u_1|)|u_2|~\txtd x+\int_D\delta_7(1+|u_1|)|z_2(\theta_t\omega)|~\txtd x\\
&=\int_D\delta_7(1+|u_1|)|u_2|~\txtd x+\delta_7\|z_2(\theta_t\omega)\|_{1}+\int_D\delta_7|u_1||z_2(\theta_t\omega)|~\txtd x\\
&\leq \int_D\delta_7(1+|u_1|)|u_2|~\txtd x+\delta_7\|z_2(\theta_t\omega)\|_{2}^2+C+\int_D\delta_7|u_1||z_2(\theta_t\omega)|~\txtd x.
\end{align*}
Now, combining the estimates for $I_2$ and $I_3$ yields
\begin{align*}
&I_2+I_3\\
&\leq \int_D\delta_7(1+|u_1|)|u_2|~\txtd x+\int_D\delta_7|u_1||z_2(\theta_t\omega)|~\txtd x\\
&\qquad +\int_D \delta_4\frac{3}{2}(|u_1|+|u_1|^{p_1+1})~\txtd x+\int_D\delta_4|u_1||u_2|~\txtd x+\int_D\delta_4|u_2||z_1(\theta_t\omega)|~\txtd x\\
&\qquad +C+ C_1 \left(\|z_2(\theta_t\omega)\|_{2}^2+\|z_1(\theta_t\omega)\|_{2}^2+\|z_1(\theta_t\omega)\|_{p_1+1}^{p_1+1}\right)\\
&\leq (\delta_4+\delta_7)\int_D(1+|u_1|)|u_2|~\txtd x+\int_D\delta_7|u_1||z_2(\theta_t\omega)|~\txtd x\\
&\qquad +\int_D \delta_4\frac{3}{2}(|u_1|+|u_1|^{p_1+1})~\txtd x+\int_D\delta_4|u_2||z_1(\theta_t\omega)|~\txtd x\\
&\qquad +C+ C_1 \left(\|z_2(\theta_t\omega)\|_{2}^2+\|z_1(\theta_t\omega)\|_{2}^2+\|z_1(\theta_t\omega)\|_{p_1+1}^{p_1+1}\right)\\
&\leq  \frac{\delta}{16}\|u_2\|_{2}^2+C_2\int_D(1+|u_1|)^2~\txtd x+\int_D\delta_7|u_1||z_2(\theta_t\omega)|~\txtd x\\
&\qquad +\int_D \delta_4\frac{3}{2}(|u_1|+|u_1|^{p_1+1})~\txtd x+\int_D\delta_4|u_2||z_1(\theta_t\omega)|~\txtd x\\
&\qquad +C+C_1 \left(\|z_2(\theta_t\omega)\|_{2}^2+\|z_1(\theta_t\omega)\|_{2}^2+\|z_1(\theta_t\omega)\|_{p_1+1}^{p_1+1}\right)\\
&= \frac{\delta}{16}\|u_2\|_{2}^2+\delta_4\frac{3}{2}\int_D \left(|u_1|+|u_1|^{p_1+1}+C_2(1+|u_1|)^2\right)~\txtd x\\
&\qquad +\int_D\delta_7|u_1||z_2(\theta_t\omega)|~\txtd x +\int_D\delta_4|u_2||z_1(\theta_t\omega)|~\txtd x\\
&\qquad +C+ C_1 \left(\|z_2(\theta_t\omega)\|_{2}^2+\|z_1(\theta_t\omega)\|_{2}^2+\|z_1(\theta_t\omega)\|_{p_1+1}^{p_1+1}\right)\\
&\leq \frac{\delta}{16}\|u_2\|_{2}^2+C_2\int_D(1+|u_1|^q)~\txtd x+\frac{\delta_1}{8}\|u_1\|_{2}^2+\frac{\delta}{16}\|u_2\|_{2}^2\\
&\qquad +C+ C_1\left(\|z_2(\theta_t\omega)\|_{2}^2+\|z_1(\theta_t\omega)\|_{2}^2+\|z_1(\theta_t\omega)\|_{p_1+1}^{p_1+1}\right),
\end{align*} 
where we have used that for $q=\max\{p_1+1,2\}<p$ there exists a constant $C_2$ such that 
\begin{equation}
C_1\left(|\xi|+|\xi|^{p_1+1}+C(1+|\xi|)^2\right)\leq C_2(|\xi|^q+1),~~~ \text{for all }\xi\in \mathbb R. 
\end{equation}
Thus, 
\begin{align*}
&I_2+I_3\\
&\leq \frac{\delta}{8}\|u_2\|_{2}^2+\frac{\delta_1}{8}\|u_1\|_{2}^2+\frac{\delta_1}{4}\|u_1\|_{p}^p\\
&\qquad+C+C_1 \left(\|z_2(\theta_t\omega)\|_{2}^2+\|z_1(\theta_t\omega)\|_{2}^2+\|z_1(\theta_t\omega)\|_{p_1+1}^{p_1+1}\right)\\
&\leq \frac{\delta}{4}\|v_2\|_{2}^2+\frac{\delta_1 3}{8}\|u_1\|_{p}^p+C\\
&\qquad +C_1 \left(\|z_2(\theta_t\omega)\|_{2}^2+\|z_1(\theta_t\omega)\|_{2}^2+\|z_1(\theta_t\omega)\|_{p_1+1}^{p_1+1}\right).
\end{align*}
Hence, in total we obtain 
\begin{align}\label{eqn:pandv}
&\frac{1}{2}\frac{\txtd }{\txtd t}(\|v_1\|^2_{2}+\|v_2\|^2_{2})
\nonumber\\&\leq -d\|\nabla v_1\|^2_{2}-\frac{\delta_1}{2}\|u_1\|_{p}^p-\delta \|v_2\|_{2}^2+\frac{\delta}{4}\|v_2\|_{2}^2+\frac{\delta_1 3}{8}\|u_1\|_{p}^p\nonumber\\
&\qquad{}+C+ C_1 \left(\|z_2(\theta_t\omega)\|_{2}^2+\|z_1(\theta_t\omega)\|_{2}^2+\|z_1(\theta_t\omega)\|_{p_1+1}^{p_1+1}+\|z_1(\theta_t\omega)\|_{p}^p\right)\nonumber\\
&= -d\|\nabla v_1\|^2_{2}-\frac{\delta_1}{8}\|u_1\|_{p}^p-\frac{3\delta }{4}\|v_2\|_{2}^2\nonumber\\
&\qquad{}+C+ C_1 \left(\|z_2(\theta_t\omega)\|_{2}^2+\|z_1(\theta_t\omega)\|_{2}^2+\|z_1(\theta_t\omega)\|_{p_1+1}^{p_1+1}+\|z_1(\theta_t\omega)\|_{p}^p\right)\nonumber\\
&\leq-\frac{d}{2}\|\nabla v_1\|^2_{2} -\frac{d}{2c}\|v_1\|^2_{2}-\frac{3\delta}{4}\|v_2\|_{2}^2\nonumber\\
&\qquad+C+ C_1 \left(\|z_2(\theta_t\omega)\|_{2}^2+\|z_1(\theta_t\omega)\|_{p}^p\right)
\end{align}
and thus 
\begin{equation}
\frac{\txtd }{\txtd t}(\|v_1\|^2_{2}+\|v_2\|^2_{2})\leq -C_2\left(\|v_1\|^2_{2}+\|v_2\|_{2}^2\right)+C+ C_1 \left(\|z_2(\theta_t\omega)\|_{2}^2+\|z_1(\theta_t\omega)\|_{p}^p\right).
\end{equation}
Now, applying Gronwall's inequality we obtain 
\begin{align}\label{eqn:2norm}
&\|v_1\|^2_{2}+\|v_2\|^2_{2}\nonumber\\&\leq \left(\|v_1^0\|^2_{2}+\|v_2^0\|^2_{2}\right)\e{-C_2t}+C_3\left(1-\e{-C_2t}\right)\nonumber\\&\qquad{}+C_1\int_0^t\e{-C_2(t-s)}\left(\|z_2(\theta_s\omega)\|_{2}^2+\|z_1(\theta_s\omega)\|_{p}^p\right)~\txtd s\nonumber\\
&\leq \left(\|v_1^0\|^2_{2}+\|v_2^0\|^2_{2}\right)\e{-C_2t}+C_3\nonumber\\&\qquad{}+ C_1\int_0^t\e{-C_2(t-s)}\left(\|z_2(\theta_s\omega)\|_{2}^2+\|z_1(\theta_s\omega)\|_{p}^p\right)~\txtd s.
\end{align}
We replace $\omega$ by $\theta_{-t}\omega$ (note the $\mathbb P$-preserving property of the MDS) and carry out a change of variables
\begin{eqnarray*}
&&\|v_1(t,\theta_{-t}\omega,v_1^0(\theta_{-t}\omega))\|^2_{2}+\|v_2(t,\theta_{-t}\omega,v_2^0(\theta_{-t}\omega))\|^2_{2}\\&&\leq\left(\|v_1^0(\theta_{-t}\omega)\|^2_{2}+\|v_2^0(\theta_{-t}\omega)\|^2_{2}\right)\e{-C_2t}+C_3\\&&\qquad{}+C_1\int_0^t\e{-C_2(t-s)}\left(\|z_2(\theta_{s-t}\omega)\|_{2}^2+\|z_1(\theta_{s-t}\omega)\|_{p}^p\right)~\txtd s
\\&&\leq\left(\|v_1^0(\theta_{-t}\omega)\|^2_{2}+\|v_2^0(\theta_{-t}\omega)\|^2_{2}\right)\e{-C_2t}+C_3\\&&\qquad{}+C_1\int_{-t}^0\e{C_2s}\left(\|z_2(\theta_s\omega)\|_{2}^2+\|z_1(\theta_s\omega)\|_{p}^p\right)~\txtd s.
\end{eqnarray*}
Now let $\cD\in \cT$ be arbitrary and $(u_1^0,u_2^0)(\theta_{-t}\omega)\in \cD(\theta_{-t}\omega)$. Then
\begin{align*}
&\|\varphi(t,\theta_{-t}\omega,(u_1^0,u_2^0)(\theta_{-t}\omega))\|_\mathbb H^2\\
&= \|v_1(t,\theta_{-t}\omega,u_1^0(\theta_{-t}\omega)-z_1(\theta_{-t}\omega))+z_1(\omega)\|_2^2\\
&\qquad +\|v_2(t,\theta_{-t}\omega,u_2^0(\theta_{-t}\omega)-z_2(\theta_{-t}\omega))+z_2(\omega)\|_2^2\\
&\leq 2\|v_1(t,\theta_{-t}\omega,u_1^0(\theta_{-t}\omega)-z_1(\theta_{-t}\omega))\|_2^2+2\|z_1(\omega)\|_2^2\\
&\qquad+2\|v_2(t,\theta_{-t}\omega,u_2^0(\theta_{-t}\omega)-z_2(\theta_{-t}\omega))\|_2^2 +2\|z_2(\omega)\|_2^2\\
&\leq 2\left(\|u_1^0(\theta_{-t}\omega)-z_1(\theta_{-t}\omega)\|^2_{2}+\|u_2^0(\theta_{-t}\omega)-z_2(\theta_{-t}\omega)\|^2_{2}\right)\e{-C_2t}\\
&\qquad{}+2C_3+2 C_1\int_{-t}^0\e{C_2s}\left(\|z_2(\theta_s\omega)\|_{2}^2+\|z_1(\theta_s\omega)\|_{p}^p\right)~\txtd s\\
&\qquad +2\|z_1(\omega)\|_2^2+2\|z_2(\omega)\|_2^2\\
&\leq 4\left(\|u_1^0(\theta_{-t}\omega)\|_2^2+\|z_1(\theta_{-t}\omega)\|^2_{2}+\|u_2^0(\theta_{-t}\omega)\|_2^2+\|z_2(\theta_{-t}\omega)\|^2_{2}\right)\e{-C_2t}\\
&\qquad{}+2C_3+2C_1\int_{-t}^0\e{C_2s}\left(\|z_2(\theta_s\omega)\|_{2}^2+\|z_1(\theta_s\omega)\|_{p}^p\right)~\txtd s\\
&\qquad +2\|z_1(\omega)\|_2^2+2\|z_2(\omega)\|_2^2.
\end{align*}
Since $(u_1^0,u_2^0)(\theta_{-t}\omega)\in \cD(\theta_{-t}\omega)$ and  since $\|z_1(\omega)\|_p^p$ ($p\geq 1$), $\|z_2(\omega)\|_2^2$ are tempered random variables, we have 
\begin{align*}
&\limsup_{t\rightarrow \infty}\left(\|u_1^0(\theta_{-t}\omega)\|_2^2+\|z_1(\theta_{-t}\omega)\|^2_{2}\right. ...\\
&\qquad \qquad \left.+\|u_2^0(\theta_{-t}\omega)\|_2^2+\|z_2(\theta_{-t}\omega)\|^2_{2}\right)\e{-C_2t}=0.
\end{align*}
Hence,
\begin{align}
&\limsup_{t\rightarrow \infty} \|\varphi(t,\theta_{-t}\omega,(u_1^0,u_2^0)(\theta_{-t}\omega))\|_\mathbb H^2\nonumber\\& \leq 2C_3+2C_1\int_{-\infty}^0\e{C_2s}\left(\|z_2(\theta_s\omega)\|_{2}^2+\|z_1(\theta_s\omega)\|_{p}^p\right)~\txtd s\nonumber \\&\qquad +2\|z_1(\omega)\|_2^2+2\|z_2(\omega)\|_2^2\nonumber\label{rho}\\
&=:\rho(\omega).
\end{align}
Due to the temperedness of $\|z_1(\omega)\|_p^p$ for $p\geq 1$ and $\|z_2(\omega)\|_2^2$, the improper integral above exists and $\rho(\omega)>0$ is a $\omega$-dependent constant. As described in Remark \ref{rem:abs}, we can define for some $\epsilon>0$ 
\begin{equation*}
\mathcal B(\omega)=B(0,\rho(\omega)+\epsilon).
\end{equation*}
Then $\cB=\{\cB(\omega)\}_\omega\in \cT$ is a $\cT$-absorbing set for the RDS $\varphi$ with finite absorption time $t_\cT(\omega)=\sup_{\cD\in \cT}t_\cD(\omega)$. 
\end{proof}

The random radius $\rho(\omega)$ depends on the restrictions imposed on the non-linearity 
and the noise. These were heavily used in Lemma~\ref{lem:abs} in order to derive the 
expression~\ref{rho} for $\rho(\omega)$. Regarding the structure of $\rho(\omega)$ we infer 
by Lemma~\ref{lem:temp} that $\rho(\omega)$ is tempered. Although we have now shown the 
existence of a bounded $\cT$-absorbing set for the RDS at hand, we need further steps. To 
show existence of a random attractor, we would like to make use of Theorem \ref{thm:atthm}, 
i.e., we have to show existence of a \textit{compact} $\cT$-absorbing set. This will be the 
goal of the next subsection. 

\subsection{Compact absorbing set}\label{subsec:compact}

The classical strategy to find a compact absorbing set in $L^2(D)$ for a reaction-diffusion 
equation is the following: Firstly, one needs to find an absorbing set in $L^2(D)$. Secondly, 
this set is used to find an absorbing set in $H^1(D)$ and due to compact embedding this automatically 
defines a compact absorbing set in $L^2(D)$. In our setting the construction of an absorbing set in 
$H^1(D)$ is more complicated as the regularizing effect of the Laplacian is missing in the second 
component of \eqref{eqn:eqv}. That is solutions with initial conditions in $L^2(D)$ will in general 
only belong to $L^2(D)$ and not to $H^1(D)$. To overcome this difficulty, we split the solution 
of the second component into two terms: one which is regular enough, in the sense that it 
belongs to $H^1(D)$ and the another one which asymptotically tends to zero. This splitting 
method has been used by several authors in the context of partly dissipative systems, see for 
instance \cite{Mar,Wang2}. Let us now explain the strategy for our setting in more detail. 
We consider the equations
\begin{equation}
\frac{\txtd v_2^1(t)}{\txtd t}
=-\sigma(x)v_2^1(t)-g(x,v_1(t)+z_1(\theta_t\omega)),~~~~v_2^1(0)=0,\label{eqn:split1}
\end{equation}
and 
\begin{equation}
\frac{\txtd v_2^2}{\txtd t}=-\sigma(x)v_2^2,~~~v_2^2(0)=v_2^0,\label{eqn:split2}
\end{equation}
then $v_2=v_2^1+v_2^2$ solves \eqref{eqn:RDS2}. Note at this point that we associate 
the initial condition $v_2^0\in L^2(D)$ to the second part. Now, let $\cD=(\cD_1,\cD_2)
\in \cT$ be arbitrary and $u^0=(u_1^0,u_2^0)\in \cD$. Then 
\begin{align*}&\varphi(t,\theta_{-t}\omega,u^0(\theta_{-t}\omega))\\&
=(u_1(t,\theta_{-t}\omega,u_1^0(\theta_{-t}\omega)),u_2(t,\theta_{-t}\omega,u_2^0(\theta_{-t}\omega)))\\
&=\left(v_1(t,\theta_{-t}\omega, v_1^0(\theta_{-t}\omega))+z_1(\omega),\right.\\
&\qquad \left.
v_2^1(t,\theta_{-t}\omega, v_2^0(\theta_{-t}\omega))+v_2^2(t,\theta_{-t}\omega, v_2^0(\theta_{-t}\omega))+z_2(\omega)\right)\\
&=\left(v_1(t,\theta_{-t}\omega, v_1^0(\theta_{-t}\omega))+z_1(\omega),\right.\\
&\qquad \left. v_2^1(t,\theta_{-t}\omega, 0)+z_2(\omega)\right)+\left(0,v_2^2(t,\theta_{-t}\omega, v_2^0(\theta_{-t}\omega))\right)\\
&=:\varphi_1(t,\theta_{-t}\omega,v_1^0(\theta_{-t}\omega))+\varphi_2(t,\theta_{-t}\omega,v_2^0(\theta_{-t}\omega))
\end{align*}
If we can show that for a certain $t^*\geq t_\cD(\omega)$ there exist tempered random variables $\rho_1(\omega)$, $\rho_2(\omega)$ such that 
\begin{align}\|v_1(t^*,\theta_{-t^*}\omega,v_1^0(\theta_{-t^*}\omega))+z_1(\omega)\|_{H^1(D)}<&\rho_1(\omega),\label{eqn:1}\\
\|v_2^1(t^*,\theta_{-t^*}\omega,0)+z_2(\omega)\|_{H^1(D)}<&\rho_2(\omega),\label{eqn:2}
\end{align}
then, because of compact embedding, we know that $\varphi_1(t^*,\theta_{-t^*}\omega, \cD_1(\theta_{-t^*}\omega))$ is a 
compact set in $\mathbb H$. If, furthermore
\begin{equation}\lim_{t\rightarrow \infty}\|v_2^2(t,\theta_{-t}\omega, v_2^0(\theta_{-t}\omega))\|_{2} =0,\label{eqn:3}
\end{equation}
then $\varphi_2(t,\theta_{-t}\omega,\cD_2(\theta_{-t}\omega))$ can be regarded as a (random) bounded perturbation and $\overline{\varphi(t,\theta_{-t}\omega,\cD(\theta_{-t}\omega))}$ is compact in $\mathbb H$ as well, see \cite[Theorem 2.1]{Temam}. 
Then, 
\begin{equation}\label{eqn:cas}\overline{\varphi(t^*,\theta_{-t^*}\omega,\cB(\theta_{-t^*}\omega))}
\end{equation}
is a compact absorbing set for the RDS $\varphi$. We will now prove the necessary estimates \eqref{eqn:1}-\eqref{eqn:3}. 

\begin{lem}
Let Assumptions \ref{ass:1} and \ref{ass:2} hold. Let $\cD_2\subset L^2(D)$ be tempered and $u_2^0\in \cD_2$. Then
\begin{align*}
\lim_{t\rightarrow \infty}\|v_2^2(t,\theta_{-t}\omega,v_2^0(\theta_{-t}\omega))\|_{2}^2=0.
\end{align*}
\end{lem}
\begin{proof}
The solution to \eqref{eqn:split2} is given by 
$$v_2^2(t)=v_2^0\e{-\sigma(x)t}$$
and thus 
\begin{align*}
&\lim_{t\rightarrow\infty}\|v_2^2(t,\theta_{-t}\omega,v_2^0(\theta_{-t}\omega))\|_{2}^2\\&=\lim_{t\rightarrow\infty}\left\|v_2^0(\theta_{-t}\omega)\e{-\sigma(x)t}\right\|_2^2\\
&\leq \lim_{t\rightarrow\infty}\|v_2^0(\theta_{-t}\omega)\|_2^2\e{-\delta t}\\
&\leq\lim_{t\rightarrow\infty}\left( \|u_2^0(\theta_{-t}\omega)\|_2^2+\|z_2(\theta_{-t}\omega)\|_2^2\right)\e{-\delta t}=0,
\end{align*}
as $u_2^0\in \cD_2$ and $\|z_2(\omega)\|_2^2$ is a tempered random variable. 
\end{proof}
 
We now prove boundedness of $v_1$ and $v_2^1$ in $H^1(D)$. Therefore we need some auxiliary estimates.  First, let us derive uniform estimates for $u_1\in L^p(D)$ and for $v_1\in H^1(D)$.
 
\begin{lem}\label{lem:aux1} Let Assumptions \ref{ass:1} and \ref{ass:2}  hold. Let $\cD_1\subset L^2(D)$ be tempered and $u_1^0\in \cD_2$. Assume $t\geq 0$, $r>0$, then 
\begin{align}\label{eqn:pnorm}
&\int_{t}^{t+r}\|u_1(s,\omega, u_1^0(\omega))\|_p^p~\txtd s\nonumber\\&\leq Cr+C_1\int_t^{t+r}\left(\|z_2(\theta_s\omega)\|_{2}^2+\|z_1(\theta_s\omega)\|_{p}^p\right)\txtd s\nonumber\\
&\qquad+\|v_1(t,\omega,v_1^0(\omega))\|^2_2+\|v_2(t,\omega,v_2^0(\omega))\|_2^2,
\end{align}
\begin{align}\label{eqn:uninabla}
&\int_{t}^{t+r}\|\nabla v_1(s,\omega, v_1^0(\omega))\|_2^2~\txtd s\nonumber\\&\leq Cr+C_1\int_t^{t+r}\left(\|z_2(\theta_s\omega)\|_{2}^2+\|z_1(\theta_s\omega)\|_{p}^p\right)\txtd s\nonumber\\
&\qquad +\|v_1(t,\omega,v_1^0(\omega))\|^2_2+\|v_2(t,\omega,v_2^0(\omega))\|_2^2,
\end{align}
where $C,C_1$ are deterministic constants. 
\end{lem}
\begin{proof}
From \eqref{eqn:pandv} we can derive
\begin{align*}
&\frac{\txtd }{\txtd t}(\|v_1\|^2_{2}+\|v_2\|^2_{2})\\
&\leq -d\|\nabla v_1\|_2^2-\frac{\delta_1}{4}\|u_1\|_{p}^p+C+C_1 \left(\|z_2(\theta_t\omega)\|_{2}^2+\|z_1(\theta_t\omega)\|_{p}^p\right),
\end{align*}
and thus by integration
\begin{align*}
&d\int_{t}^{t+r}\| \nabla v_1(s,\omega,v_1^0(\omega))\|_2^2~\txtd s+\frac{\delta_1}{4}\int_{t}^{t+r}\|u_1(s,\omega,u_1^0(\omega))\|_p^p~\txtd s\\
&\leq Cr+C_1\int_t^{t+r}\left(\|z_2(\theta_s\omega)\|_{2}^2+\|z_1(\theta_s\omega)\|_{p}^p\right)\txtd s\\
&\qquad+\|v_1(t,\omega,v_1^0(\omega))\|^2_2+\|v_2(t,\omega,v_2^0(\omega))\|_2^2.
\end{align*}
The two statements of the lemma follow directly from this estimate.  
\end{proof}

\begin{lem}\label{lem:2p}
Let Assumptions  \ref{ass:1} and \ref{ass:2}  hold. Let $\cD_1\subset L^2(D)$ be tempered and $u_1^0\in \cD_1$. Assume  
$t\geq r$, then 
\begin{align}
&\int_{t}^{t+r} \|u_1(s,\omega,u_1^0(\omega))\|_{2p-2}^{2p-2}~\txtd s\nonumber \\
&\leq C_6r+\int_{t-r}^{t+r}C_2\|z_1(\theta_s\omega)\|_{p^2-p}^{p^2-p}+C_3\|z_2(\theta_s\omega)\|_2^2+C_4\|v_2(s,\omega,v_2^0(\omega))\|_2^2~\txtd s\nonumber\\
&\qquad+C_5\|v_1(t-r,\omega,v_1^0(\omega))\|^2_2+C_5\|v_2(t-r,\omega,v_2^0(\omega))\|_2^2,\label{eqn:some}
\end{align}
where $C_2,C_3,C_4,C_5,C_6$ are deterministic constants.  
\end{lem}
\begin{proof}
Remember that $v_1$ satisfies equation \eqref{eqn:RDS1}. 
Multiplying this equation by $|v_1|^{p-2}v_1$ and integrating over $D$ yields
\begin{align*}
&\frac{1}{p}\frac{\txtd }{\txtd t}\int_D|v_1|^{p}~\txtd x\\&=d\int_D\Delta v_1(t)|v_1|^{p-2}v_1~\txtd x-\int_D h(x,v_1(t)+z_1(\theta_t\omega))|v_1|^{p-2}v_1~\txtd x\\
&\qquad{}-\int_D f(x,v_1(t)+z_1(\theta_t\omega),v_2(t)+z_2(\theta_t\omega))|v_1|^{p-2}v_1 ~\txtd x\\
&=  -d(p-1)\int_D |\nabla v_1|^2|v_1|^{p-2}~\txtd x-\int_D h(x,v_1(t)+z_1(\theta_t\omega))|v_1|^{p-2}v_1~\txtd x\\
&\qquad{}-\int_D f(x,v_1(t)+z_1(\theta_t\omega),v_2(t)+z_2(\theta_t\omega))|v_1|^{p-2}v_1 ~\txtd x\\
&\leq  -\int_D \left(\frac{\delta_1}{2^p}|v_1|^p-C-C_1(|z_1(\theta_t\omega)|^2+|z_1(\theta_t\omega)|^p)\right)|v_1|^{p-2}~\txtd x\\
&\qquad{}+\int_D |f(x,v_1(t)+z_1(\theta_t\omega),v_2(t)+z_2(\theta_t\omega))||v_1|^{p-2}v_1 ~\txtd x\\
&\leq  -\int_D \frac{\delta_1}{2^p}|v_1|^{2p-2}~\txtd x+C\int_D|v_1|^{p-2}~\txtd x\\
&\qquad + C_1\int_D (|z_1(\theta_t\omega)|^2+|z_1(\theta_t\omega)|^p)|v_1|^{p-2}~\txtd x\\
&\qquad+\int_D \delta_4 (1+|v_1+z_1(\theta_t\omega)|^{p_1}+|v_2+z_2(\theta_t\omega)|)|v_1|^{p-2}v_1 ~\txtd x\\ 
&\leq  -\int_D \frac{\delta_1}{2^p}|v_1|^{2p-2}~\txtd x+C\int_D|v_1|^{p-2}~\txtd x+C_1\int_D|v_1|^{p-1}~\txtd x\\
&\qquad +C_2\int_D(|z_1(\theta_t\omega)|^{2p-2}+|z_1(\theta_t\omega)|^{p^2-p}) ~\txtd x+\int_D \delta_4 \left(|v_1|^{p-1}...\right.\\
&\qquad \qquad +C_3\left(|v_1|^{p_1+p-1}+|z_1(\theta_t\omega)|^{p_1}|v_1|^{p-1}+|v_2||v_1|^{p-1}...\right.\\
&\qquad \qquad \left.\left.+|z_2(\theta_t\omega)||v_1|^{p-1}\right)\right) ~\txtd x\\
&\leq  -\int_D \frac{\delta_1}{2^p}|v_1|^{2p-2}~\txtd x+\frac{\delta_1}{2^p 4}\int_D|v_1|^{2p-2}~\txtd x+C_6\\
&\qquad+C_2\int_D(|z_1(\theta_t\omega)|^{2p-2}+|z_1(\theta_t\omega)|^{p^2-p}) ~\txtd x\\
&\qquad{}+\int_D C_3 (|z_1(\theta_t\omega)|^{p_1}|v_1|^{p-1}+|v_2||v_1|^{p-1}+|z_2(\theta_t\omega)||v_1|^{p-1}) ~\txtd x,
\end{align*}
where we have used condition \eqref{eqn:condf},  the relations $p-1,p-2,p_1+p-1<2p-2$ and the inequality 

\begin{equation*}
h(x,v_1+z_1)v_1\geq \frac{\delta_1}{2^p}|v_1|^p-C-C_1(|z_1|^2+|z_1|^p),
\end{equation*}
that can be proved by using conditions \eqref{eqn:condh} and \eqref{eqn:condhnew}
\begin{align*}
h(x,v_1+z_1)v_1&=h(x,v_1+z_1)(v_1+z_1)-h(x,v_1+z_1)z_1\\
&\geq \delta_1 |v_1+z_1|^p-\delta_3-|h(x,v_1+z_1)||z_1|\\
&\geq \delta_1|v_1+z_1|^p-\delta_3-(\delta_8+\delta_8|v_1+z_1|^{p-1})|z_1|\\
&\geq \delta_1|v_1+z_1|^p-C-C_1|z_1|^2-\delta_1/2|v_1+z_1|^p-C_2|z_1|^p\\
&=\frac{\delta_1}{2}|v_1+z_1|^p-C-C_1(|z_1|^2+|z_1|^p)\\
&\geq \frac{\delta_1}{2}||v_1|-|z_1||^p-C-C_1(|z_1|^2+|z_1|^p)\\
&\geq \frac{\delta_1}{2^p}|v_1|^p-C-C_1(|z_1|^2+|z_1|^p).
\end{align*}
Hence we have
\begin{align*}
&\frac{1}{p}\frac{\txtd }{\txtd t}\int_D|v_1|^{p}~\txtd x+\int_D  \frac{3}{4}\frac{\delta_1}{2^p}|v_1|^{2p-2}~\txtd x\\&\leq C_6+C_2\int_D(|z_1(\theta_t\omega)|^{2p-2}+|z_1(\theta_t\omega)|^{p^2-p}) ~\txtd x\\
&\qquad +\int_D C_3 (|z_1(\theta_t\omega)|^{p_1}+|v_2|+|z_2(\theta_t\omega)|)|v_1|^{p-1} ~\txtd x
\\&\leq C_6+C_2\int_D(|z_1(\theta_t\omega)|^{2p-2}+|z_1(\theta_t\omega)|^{p^2-p}) ~\txtd x+\int_D \frac{1}{4}\frac{\delta_1}{2^p} |v_1|^{2p-2}~\txtd x \\&\qquad{}+\int_D C_3 (|z_1(\theta_t\omega)|^{p_1}+|v_2|+|z_2(\theta_t\omega)|)^2 ~\txtd x
\end{align*}
and thus
\begin{align}\label{eqn:someeq}
&\frac{1}{p}\frac{\txtd }{\txtd t}\int_D|v_1|^{p}~\txtd x+\int_D \frac{1}{2}\frac{\delta_1}{2^p}|v_1|^{2p-2}~\txtd x\nonumber\\
&\leq C_6+C_2\int_D(|z_1(\theta_t\omega)|^{2p-2}+|z_1(\theta_t\omega)|^{p^2-p}) ~\txtd x \nonumber\\
&\qquad +\int_D C_3 (|z_1(\theta_t\omega)|^{2p_1}+|v_2(t)|^2+|z_2(\theta_t\omega)|^2 )~\txtd x.
\end{align}
We arrive at the following inequality 
\begin{equation}\label{eqn:someinequality}
\frac{1}{p}\frac{\txtd }{\txtd t} \|v_1\|_p^p+\frac{\delta_1}{2^{p+1}}\|v_1\|^{2p-2}_{2p-2}\leq C_6+C_2\|z_1(\theta_t\omega)\|_{p^2-p}^{p^2-p}+C_3\|z_2(\theta_t\omega)\|_2^2+C_3\|v_2\|_2^2
\end{equation}
and hence
\begin{equation}\label{eqn:ers}
\frac{\txtd }{\txtd t} \|v_1\|_p^p\leq C_6+C_2\|z_1(\theta_t\omega)\|_{p^2-p}^{p^2-p}+C_3\|z_2(\theta_t\omega)\|_2^2+C_3\|v_2\|_2^2-\frac{\delta_1}{2^{p+1}}\|v_1\|^{p}_{p}.
\end{equation}
With \eqref{eqn:pnorm} we have 
\begin{align*}
&\int_t^{t+r}\|v_1(s,\omega,v_1^0(\omega))\|_p^p~\txtd s\\
&=\int_t^{t+r}\|u_1(s,\omega,v_1^0(\omega))-z_1(\theta_s\omega)\|_p^p~\txtd s\\
&\leq Cr+C_1\int_t^{t+r}\left(\|z_2(\theta_s\omega)\|_{2}^2+\|z_1(\theta_s\omega)\|_{p}^p\right)\txtd s\\
&\qquad +C_2\|v_1(t,\omega,v_1^0(\omega))\|^2_2+C_2\|v_2(t,\omega,v_2^0(\omega))\|_2^2.
\end{align*}
Thus by applying the uniform Gronwall Lemma to~\eqref{eqn:ers} we have 
\begin{align}\label{eqn:eq10}
&\|v_1(t+r,\omega,v_1^0(\omega))\|_p^p\nonumber\\& \leq rC_6+\int_t^{t+r} C_2\|z_1(\theta_s\omega)\|_{p^2-p}^{p^2-p}+C_3\|z_2(\theta_s\omega)\|_2^2+C_4\|v_2(s,\omega,v_2^0(\omega))\|_2^2~\txtd s\nonumber\\&\qquad +C_5\|v_1(t,\omega,v_1^0(\omega))\|^2_2+C_5\|v_2(t,\omega,v_2^0(\omega))\|_2^2.
\end{align}
Now integrating \eqref{eqn:someinequality} between $t$ and $t+r$ yields
\begin{align*}
&\int_t^{t+r}\|v_1(s,\omega,v_1(\omega))\|^{2p-2}_{2p-2}~\txtd s\\
&\leq C_6r+\int_t^{t+r}C_2\|z_1(\theta_s\omega)\|_{p^2-p}^{p^2-p}+C_3\|z_2(\theta_s\omega)\|_2^2+C_3\|v_2(s,\omega,v_2^0(\omega))\|_2^2~\txtd s\\&\qquad +C\|v_1(t,\omega,v_1^0(\omega))\|_p^p
\end{align*}
and thus for $t\geq r$ using \eqref{eqn:eq10}
\begin{align*}
&\int_t^{t+r}\|v_1(s,\omega,v_1(\omega))\|^{2p-2}_{2p-2}~\txtd s\\
&\leq C_6r+\int_{t-r}^{t+r}C_2\|z_1(\theta_s\omega)\|_{p^2-p}^{p^2-p}+C_3\|z_2(\theta_s\omega)\|_2^2+C_4\|v_2(s,\omega,v_2^0(\omega))\|_2^2~\txtd s\\&\qquad +C_5\|v_1(t-r,\omega,v_1^0(\omega))\|^2_2+C_5\|v_2(t-r,\omega,v_2^0(\omega))\|_2^2.
\end{align*}
In total this leads to 
\begin{align*}
&\int_t^{t+r}\|u_1(s,\omega,v_1(\omega))\|^{2p-2}_{2p-2}~\txtd s\\&\leq C_6r+\int_{t-r}^{t+r}C_2\|z_1(\theta_s\omega)\|_{p^2-p}^{p^2-p}+C_3\|z_2(\theta_s\omega)\|_2^2+C_4\|v_2(s,\omega,v_2^0(\omega))\|_2^2~\txtd s\\&\qquad{}+C_5\|v_1(t-r,\omega,v_1^0(\omega))\|^2_2+C_5\|v_2(t-r,\omega,v_2^0(\omega))\|_2^2\\
&\qquad +\int_t^{t+r}\|z_1(\theta_s\omega)\|_{2p-2}^{2p-2}~\txtd s\\
&\leq C_6r+\int_{t-r}^{t+r}C_2\|z_1(\theta_s\omega)\|_{p^2-p}^{p^2-p}+C_3\|z_2(\theta_s\omega)\|_2^2+C_4\|v_2(s,\omega,v_2^0(\omega))\|_2^2~\txtd s\\&\qquad{}+C_5\|v_1(t-r,\omega,v_1^0(\omega))\|^2_2+C_5\|v_2(t-r,\omega,v_2^0(\omega))\|_2^2,
\end{align*}
and this finishes the proof.
\end{proof}

One can also use appropriate shifts within the integrals on the left hand sides in \eqref{eqn:pnorm}, \eqref{eqn:uninabla}, \eqref{eqn:some} to obtain simpler forms of the $\omega$-dependent constants on the right hand side, see for instance \cite[Lemma 4.3, 4.4]{Wang}. More precisely, in case of \eqref{eqn:pnorm} one can for instance obtain an estimate of the form 
\begin{align*}
\int\limits_{t}^{t+r}\|u_{1}(s,\theta_{-t-r}\omega,u^{0}_{1}(\theta_{-t-r}\omega))\|^p_{p} \leq c (1+\widetilde{\rho}(\omega)), 
\end{align*}
where $\tilde \rho(\omega)$ is a random constant.  Nevertheless such estimates hold for every $\omega$, independent of the shift that one inserts inside the integral on the left hand side. Without the appropriate shifts on the left hand sides, as in the lemmas above, the constants on the right hand sides depend on the shift. Next, we are going to show the boundedness of $v_1$ in $H^1(D)$. 

\begin{lem} \label{lem:HV}Let Assumptions 2.1 and 2.2 hold. Let $\cD=(\cD_1,\cD_2)\in \cT$ and $u^0\in \cD$. Assume  $t\geq t_\cD(\omega)+2r$ for some $r>0$ then 
\begin{equation}
\|\nabla v_1(t,\theta_{-t}\omega,v_1^0(\theta_{-t}\omega))\|_2^2\leq \rho_1(\omega),
\end{equation}
where $\rho_1(\omega)$ is a tempered random variable. 
\end{lem}
\begin{proof}
Remember that $v_1$ satisfies the equation \eqref{eqn:RDS1} 
and thus 
\begin{align*}
&\frac{1}{2}\frac{\txtd}{\txtd t}\|\nabla v_1\|_{2}^2=\left\langle \frac{\txtd }{\txtd t}v_1,-\Delta v_1\right\rangle\\
&=\langle d\Delta v_1-h(x,v_1+z_1(\theta_t\omega))-f(x,v_1+z_1(\theta_t\omega),v_2+z_2(\theta_t\omega)),-\Delta v_1\rangle \\
&=-d \|\Delta v_1\|_{2}^2+\langle h(x,v_1+z_1(\theta_t\omega)),\Delta v_1\rangle\\
&\qquad +\langle f(x,v_1+z_1(\theta_t\omega),v_2+z_2(\theta_t\omega)),\Delta v_1\rangle\\
&\leq  -d\|\Delta v_1\|_2^2+\int_D \delta_8(1+|u_1|^{p-1})|\Delta v_1|~\txtd x\\
&\qquad +\int_D \delta_4(1+|u_1|^{p_1}+|u_2|)|\Delta v_1|~\txtd x \\
&\leq -d\|\Delta v_1\|_2^2+C\int_D (2+|u_1|^{p-1}+|u_1|^{p_1}+|u_2|)|\Delta v_1|~\txtd x \\
&\leq -\frac{d}{2}\|\Delta v_1\|_2^2+C\int_D (1+|u_1|^{p-1}+|u_1|^{p_1}+|u_2|)^2~\txtd x \\
&\leq  -\frac{d}{2}\|\Delta v_1\|_2^2+C\int_D (1+|u_1|^{2p-2}+|u_2|^2)~\txtd x \\
&= -\frac{d}{2}\|\Delta v_1\|_2^2+C_1+ C\|u_1\|_{2p-2}^{2p-2}+ C \| u_2\|_2^2\\
&\leq  -\frac{dc}{2}\|\nabla v_1\|_2^2+C_1+ C\|u_1\|_{2p-2}^{2p-2}+C \| u_2\|_2^2. 
\end{align*}
 We want to apply the uniform Gronwall Lemma now. Therefore, note
\begin{align*}\label{eqn:gr}
\frac{\txtd}{\txtd t}\underbrace{\|\nabla v_1(t,\omega,v_1^0(\omega))\|_{2}^2}_{:=y(t)}
\leq&\underbrace{-dc}_{:=g(t)}\|\nabla v_1(t,\omega,v_1^0(\omega))\|_2^2\\
& +\underbrace{C_1+ C\|u_1(t,\omega,u_1^0(\omega))\|_{2p-2}^{2p-2}+C \| u_2(t,\omega,u_2^0(\omega))\|_2^2}_{:=h(t)}.
\end{align*}
We calculate
\begin{equation}
\int_{t}^{t+r}g(s)~\txtd s\leq 0
\end{equation}
and 
\begin{align*}
&\int_{t}^{t+r}\| \nabla v_1(s,\omega,v_1^0(\omega))\|_2^2~\txtd s\\
&\leq Cr+C_1\int_t^{t+r}\left(\|z_2(\theta_s\omega)\|_{2}^2+\|z_1(\theta_s\omega)\|_{p}^p\right)\txtd s\\
&\qquad +C_2\left(\|v_1(t,\omega,v_1^0(\omega))\|^2_2+\|v_2(t,\omega,v_2^0(\omega))\|_2^2\right.
\end{align*}
where we have applied Lemma \ref{lem:aux1}. 
By Lemma \ref{lem:2p} for $t\geq r$
\begin{align*}
&\int_t^{t+r}\|u_1(s,\omega,u_1^0(\omega))\|^{2p-2}_{2p-2}~\txtd s
\\
&\leq C_6r+\int_{t-r}^{t+r}C_2\|z_1(\theta_s\omega)\|_{p^2-p}^{p^2-p}+C_3\|z_2(\theta_s\omega)\|_2^2+C_4\|u_2(s,\omega,v_2^0(\omega))\|_2^2~\txtd s\\
&\qquad +C_5\|v_1(t-r,\omega,v_1^0(\omega))\|^2_2+C_5\|v_2(t-r,\omega,v_2^0(\omega))\|_2^2.
\end{align*}
Now, the uniform Gronwall Lemma yields for $t\geq r$
\begin{align*}
&\|\nabla v_1(t+r,\omega,v_1^0(\omega))\|_2^2\\&\leq C+C_1\int_t^{t+r}\left(\|z_2(\theta_s\omega)\|_{2}^2+\|z_1(\theta_s\omega)\|_{p}^p\right)\txtd s\\
&\qquad +C_2\left(\|v_1(t,\omega,v_1^0(\omega))\|^2_2+\|v_2(t,\omega,v_2^0(\omega))\|_2^2\right)\\
&\qquad{}+C_3\int_{t-r}^{t+r}\|z_1(\theta_s\omega)\|_{p^2-p}^{p^2-p}+\|z_2(\theta_s\omega)\|_2^2+\|u_2(s,\omega,v_2^0(\omega))\|_2^2~\txtd s\\
&\qquad{}+C_4\left(\|v_1(t-r,\omega,v_1^0(\omega))\|^2_2+\|v_2(t-r,\omega,v_2^0(\omega))\|_2^2\right)\\
&\qquad +C_5\int_t^{t+r}\|u_2(s,\omega,u_2^0(\omega))\|_2^2~\txtd s\\
&\leq C+C_1\int_{t-r}^{t+r}\|u_2(s,\omega,u_2^0(\omega))\|_2^2~\txtd s\\
&\qquad +C_2\int_{t-r}^{t+r}\|z_1(\theta_s\omega)\|_{p^2-p}^{p^2-p}+\|z_2(\theta_s\omega)\|_2^2~\txtd s \\
&\qquad +C_3\left(\|v_1(t,\omega,v_1^0(\omega))\|^2_2 
+\|v_2(t,\omega,v_2^0(\omega))\|_2^2
...\right.\\
&\qquad \qquad  \left.+\|v_1(t-r,\omega,v_1^0(\omega))\|^2_2+\|v_2(t-r,\omega,v_2^0(\omega))\|_2^2\right).
\end{align*}
That is, for $t\geq 0$ we have
\begin{align*}
&\|\nabla v_1(t+2r,\omega,v_1^0(\omega))\|_2^2\\
&\leq C+C_1\int_{t}^{t+2r}\|v_2(s,\omega,u_2^0(\omega))\|_2^2~\txtd s\\
&\qquad +C_2\int_{t}^{t+2r}\|z_1(\theta_s\omega)\|_{p^2-p}^{p^2-p}+\|z_2(\theta_s\omega)\|_2^2~\txtd s\\&\qquad{}+C_3\left(\|v_1(t+r,\omega,v_1^0(\omega))\|^2_2+\|v_2(t+r,\omega,v_2^0(\omega))\|_2^2...\right.\\
&\qquad \qquad \left.+\|v_1(t,\omega,v_1^0(\omega))\|^2_2+\|v_2(t,\omega,v_2^0(\omega))\|_2^2\right).
\end{align*} 
Let us recall that our goal is to find a $t^*\geq t_\cD(\omega)$ such that \eqref{eqn:1} holds. 
Now assume that $t\geq t_\cD(\omega)$. We replace $\omega$ by $\theta_{-t-2r}\omega$ (again note the $\mathbb P$-preserving property of the MDS), then 
\begin{align*}
&\|\nabla v_1(t+2r,\theta_{-t-2r}\omega,v_1^0(\theta_{-t-2r}\omega))\|_2^2\\
&\leq C+C_1\int_{t}^{t+2r}\|v_2(s,\theta_{-t-2r}\omega,u_2^0(\theta_{-t-2r}\omega))\|_2^2~\txtd s\\&\qquad{}+C_2\int_{t}^{t+2r}\|z_1(\theta_{s-t-2r}\omega)\|_{p^2-p}^{p^2-p}+\|z_2(\theta_{s-t-2r}\omega)\|_2^2~\txtd s\\&\qquad{}+C_3\left(\|v_1(t+r,\theta_{-t-2r}\omega,v_1^0(\theta_{-t-2r}\omega))\|^2_2\right. ...\\
&\qquad  \qquad +\|v_2(t+r,\theta_{-t-2r}\omega,v_2^0(\theta_{-t-2r}\omega))\|_2^2...\\
&\qquad{}\qquad{}+ \|v_1(t,\theta_{-t-2r}\omega,v_1^0(\theta_{-t-2r}\omega))\|^2_2...\\
&\qquad \qquad \left.+\|v_2(t,\theta_{-t-2r}\omega,v_2^0(\theta_{-t-2r}\omega))\|_2^2\right).
\end{align*} 
As $t\geq t_\cD(\omega)$ we know by the absorption property  that there exists a $\tilde \rho(\omega)$ such that 
$$\|v_1(t,\theta_{-t}\omega,v_1^0(\theta_{-t}\omega))\|_2^2\leq \tilde \rho(\omega),$$
and thus replacing $\omega$ by $\theta_{-2r}\omega$
$$\|v_1(t,\theta_{-t-2r}\omega,v_1^0(\theta_{-t-2r}\omega))\|_2^2\leq \tilde \rho(\theta_{-2r}\omega).$$
Similarly, we know that 
$$\|v_1(t+r,\theta_{-t-r}\omega,v_1^0(\theta_{-t-r}\omega))\|_2^2\leq \tilde \rho(\theta_{-r}\omega),$$
and thus by replacing  $\omega$ by $\theta_{-r}\omega$
$$\|v_1(t+r,\theta_{-t-2r}\omega,v_1^0(\theta_{-t-2r}\omega))\|_2^2\leq \tilde \rho(\theta_{-2r}\omega).$$
The same arguments hold for $v_2$.  Furthermore, as $t\geq t_\cD(\omega)$ and we know from Lemma \ref{lem:abs} that there exists a tempered random variable $\hat\rho(\omega)$ such that for $s\in (t,t+2r)$ 
\begin{align*}
\|v_{2}(s,\theta_{-s}\omega,u_2^0(\theta_{-s}\omega))\|^{2}_{2} \leq \hat{\rho}(\omega)
\end{align*}
and thus 
\begin{align*}
&\int\limits_{t}^{t+2r}\|v_{2}(s,\theta_{-t-2r}\omega,u_2^0(\theta_{-t-2r}\omega))\|^{2}_{2} \txtd s \\
&\leq \int\limits_{t}^{t+2r} \hat{\rho} (\theta_{s-t-2r}\omega)~\txtd s=\int\limits_{0}^{2r}\hat{\rho} (\theta_{\tau-2r}\omega)~\txtd \tau =\int\limits_{-2r}^{0}\hat{\rho} (\theta_{y}\omega)\txtd y.
\end{align*}
With similar substitutions in the integral over $\|z_1(\theta_{s-t-2r}\omega)\|_{p^2-p}^{p^2-p}$ and\newline $\|z_2(\theta_{s-t-2r}\omega)\|_2^2$ we arrive at 
\begin{align*}
&\|\nabla v_1(t+2r,\theta_{-t-2r}\omega,v_1^0(\theta_{-t-2r}\omega))\|_2^2\\
&\leq C+C_1\int\limits_{-2r}^{0}\hat{\rho} (\theta_{y}\omega)\txtd y+C_2\int_{-2r}^{0}\|z_1(\theta_{y}\omega)\|_{p^2-p}^{p^2-p}+\|z_2(\theta_{y}\omega)\|_2^2~\txtd y\\
&\qquad +C_3 \tilde \rho(\theta_{-2r}\omega),
\end{align*} 
where the right hand side is independent of $t$. Due to the temperedness of all terms involved, they can be combined into one tempered random variable $\rho_1(\omega)$ such that for $t\geq t_\cD(\omega)+2r=:t^*$ we have 
\begin{equation*}
\|\nabla v_1(t,\theta_{-t}\omega,v_1^0(\theta_{-t}\omega))\|_2^2\leq \rho_1(\omega),
\end{equation*}
this concludes the proof. 
\end{proof}

We are now able to prove the boundedness of the first term of $v_2$ in $H^1(D)$. 

\begin{lem}\label{lem:firstcomp}
Let Assumptions  \ref{ass:1} and \ref{ass:2}  hold. Let $\cD=(\cD_1,\cD_2)\in \cT$ and $u^0\in \cD$. Assume $t\geq t_\cD(\omega)+2r$ for some  $r>0$. Then we have 
\begin{equation}
\|\nabla v_2^1(t,\theta_{-t}\omega,0)\|_2^2\leq \rho_2(\omega),
\end{equation}
where $\rho_2(\omega)$ is a tempered random variable. 
\end{lem}
\begin{proof}
Remember that $v_2^1$ satisfies the equation \eqref{eqn:split1}
and thus 
\begin{align*}
\frac{1}{2}\frac{\txtd}{\txtd t}\|\nabla v_2^1\|_{2}^2&=\langle \frac{\txtd }{\txtd t}v_2^1,-\Delta v_2^1\rangle\\
&=\langle -\sigma(x)v_2^1-g(x,v_1+z_1),-\Delta v_2^1\rangle \\
&=\underbrace{\langle \sigma(x)v_2^1, \Delta v_2^1\rangle}_{=:L_1}+\underbrace{\langle g(x,v_1+z_1),\Delta v_2^1\rangle}_{=:L_2}.
\end{align*}
We estimate $L_1$ and $L_2$ separately 
\begin{align*}
L_1&=\int_D\sigma(x)v_2^1\Delta v_2^1 \txtd x\\
&=-\int_D\nabla(\sigma(x)v_2^1)\cdot \nabla v_2^1\txtd x\\
&\leq -\delta \|\nabla v_2^1\|_{2}^2-\int_D \nabla \sigma(x)v_2^1\cdot \nabla v_2^1\txtd x,
\end{align*}
and
\begin{align*}
L_2&=\int_Dg(x,v_1+z_1)\Delta v_2^1~\txtd x=-\int_D\nabla g(x,v_1+z_1)\cdot \nabla v_2^1~\txtd x\\
&=-\int_D \left(\nabla g(x,v_1+z_1)+\partial_\xi g(x,v_1+z_1)\nabla(v_1+z_1)\right)\cdot \nabla v_2^1~\txtd x,
\end{align*}
where in the last equation the gradient is to be understood as 
$$\nabla g(x,v_1+z_1)=(\partial_{x_1}g(x,v_1+z_1),...,\partial_{x_n}g(x,v_1+z_1))^\top.$$
Hence, 
\begin{align*}
&\frac{\txtd}{\txtd t}\|\nabla v_2^1\|_{2}^2+2\delta \|\nabla v_2^1\|_{2}^2\\&\leq 2\int_D\left|\nabla \sigma(x)v_2^1+\nabla g(x,v_1+z_1)+\partial_\xi g(x,v_1+z_1)\nabla(v_1+z_1)\right| |\nabla v_2^1|~\txtd x\\
&\leq \frac{1}{\delta} \int_D\left|\nabla \sigma(x)v_2^1+\nabla g(x,v_1+z_1)+\partial_\xi g(x,v_1+z_1)\nabla (v_1+z_1)\right|^2~\txtd x\\
&\qquad +\delta\|\nabla v_2^1\|_{2}^2
\end{align*}
and further with \eqref{eqn:condg}
\begin{align*}
&\frac{\txtd}{\txtd t}\|\nabla v_2^1\|_{2}^2+\delta \|\nabla v_2^1\|_{2}^2\\
&\leq \frac{1}{\delta} \int_D\sum_{i=1}^n\left(|\partial_{x_i} \sigma(x)v_2^1|+|\partial_{x_i}g(x,v_1+z_1)|\right....\\
&\qquad \left.+|\partial_\xi g(x,v_1+z_1)\partial_{x_i}(v_1+z_1)|\right)^2~\txtd x\\
&\leq \frac{1}{\delta} \int_D\sum_{i=1}^n\left(C|v_2^1|+\delta_5(1+|v_1+z_1|)+\delta_5|\partial_{x_i}(v_1+z_1)|\right)^2~\txtd x\\
&\leq \frac{2}{\delta} (C+\delta_5)^2 n \int_D\left(|v_2^1|+1+|v_1+z_1|\right)^2~\txtd x+\frac{2\delta_5^2}{\delta} \int_D\sum_{i=1}^n |\partial_{x_i}(v_1+z_1)|^2~\txtd x\\
&=\frac{2}{\delta} (C+\delta_5)^2 n \int_D\left(|v_2^1|+1+|v_1+z_1|\right)^2~\txtd x+\frac{2\delta_5^2}{\delta} \|\nabla(v_1+z_1)\|_{2}^2\\
&\leq C_1+C_2(\|v_2^1\|_2^2+\|v_1\|_2^2+\|z_1\|_2^2)+C_3(\|\nabla v_1\|_2^2+\|\nabla z_1\|_2^2).
\end{align*}
where $C:=\max_{1\leq i\leq n}\max_{x\in \overline D}| \partial_{x_i}\sigma(x)|$. Next, we apply Gronwall's inequality while taking the initial condition into account and we obtain for $t\geq 0$
\begin{align}\label{eqn:zuz}
\|\nabla v_2^1\|_2^2&\leq \int_0^t \left[C_1+C_2(\|v_2^1\|_2^2+\|v_1\|_2^2+\|z_1\|_2^2)+C_3(\|\nabla v_1\|_2^2+\|\nabla z_1\|_2^2)\right] ....\nonumber \\
&\qquad \times \e{(s-t)\delta}~\txtd s.
\end{align}
We have from \eqref{eqn:pandv} the following equation 
\begin{equation}\label{eqn:koko}\frac{\txtd }{\txtd t}(\|v_1\|_2^2+\|v_2\|_2^2)+M(\|v_1\|_2^2+\|v_2\|_2^2)+d\|\nabla v_1\|_2^2\leq \hat C+\tilde C(\|z_2(\theta_t\omega)\|_2^2+\|z_1(\theta_t\omega)\|_p^p),
\end{equation}
where $M=\min\{d/c,\delta\}$ and certain constants $\hat C,\tilde C$.
We multiply \eqref{eqn:koko} by $\exp(Mt)$ and integrate between $0$ and $t$
\begin{align*}
&\int_0^t \exp(Ms)\frac{\txtd }{\txtd s} (\|v_1\|_2^2+\|v_2\|_2^2)\txtd s+M\int_0^t \exp(Ms) (\|v_1\|_2^2+\|v_2\|_2^2) \txtd s\\
&\qquad{} +d \int_0^t \exp(Ms)\|\nabla v_1\|_2^2\txtd s\\
&\leq \int_0^t \hat C\exp(Ms)\txtd s+\tilde C\int_0^t \exp(Ms) (\|z_2(\theta_s\omega)\|_2^2+\|z_1(\theta_s\omega)\|_p^p)\txtd s.
\end{align*}
This yields 
\begin{align}\label{eqn:zuzu}
&\int_0^t\exp(M(s-t))\|\nabla v_1(s,\omega,v_1^0(\omega))\|_2^2\txtd s\nonumber\\
&\leq \frac{1}{d}\exp(-Mt)(\|v_1^0(\omega)\|_2^2+\|v_2^0(\omega)\|_2^2) +\hat C\nonumber\\
&\qquad{}+\tilde C\int_0^t \exp(M(s-t)) (\|z_2(\theta_s\omega)\|_2^2+\|z_1(\theta_s\omega)\|_p^p)\txtd s,
\end{align}
as well as
\begin{align*}
&\|v_1(t,\omega,v_1^0(\omega))\|^2_{2}+\|v_2(t,\omega,v_2^0(\omega))\|^2_{2}\nonumber\\
&\leq \left(\|v_1^0(\omega)\|^2_{2}+\|v_2^0(\omega)\|^2_{2}\right)\e{-Mt}+\hat C\nonumber\\&\qquad{}+ \tilde C\int_0^t\e{M(s-t)}\left(\|z_2(\theta_s\omega)\|_{2}^2+\|z_1(\theta_s\omega)\|_{p}^p\right)~\txtd s.
\end{align*}
In particular, from the last estimate we obtain 
\begin{align}\label{eqn:wowo}
&\int_0^{t_\cD(\omega)}(\|v_1(s,\theta_{-t}\omega,v_1^0(\theta_{-t}\omega)\|_2^2+\|v_2(s,\theta_{-t}\omega,v_2^0(\theta_{-t}\omega))\|_2^2)\exp(M(s-t))\txtd s\nonumber \\
&\leq \int_0^{t_\cD(\omega)} \left(\|v_1^0(\theta_{-t}\omega)\|^2_{2}+\|v_2^0(\theta_{-t}\omega)\|^2_{2}\right)\e{-Mt} \txtd s+\hat C\int_0^{t_\cD(\omega)}\exp(M(s-t))\txtd s\nonumber\\&\qquad{}+ \tilde C\int_0^{t_\cD(\omega)}\int_0^s\e{M(\tau-t)}\left(\|z_2(\theta_{\tau-t}\omega)\|_{2}^2+\|z_1(\theta_{\tau-t}\omega)\|_{p}^p\right)~\txtd \tau \txtd s\nonumber\\
&\leq \left(\|v_1^0(\theta_{-t}\omega)\|^2_{2}+\|v_2^0(\theta_{-t}\omega)\|^2_{2}\right)\e{-Mt} t_\cD(\omega)+\hat C\nonumber\\&\qquad{}+ \tilde C t_\cD(\omega) \int_0^{t_\cD(\omega)}\e{M(\tau-t)}\left(\|z_2(\theta_{\tau-t}\omega)\|_{2}^2+\|z_1(\theta_{\tau-t}\omega)\|_{p}^p\right)~\txtd \tau .
\end{align}
where we have replaced $\omega$ by $\theta_{-t}\omega$ after integrating and used that $t\geq t_\cD(\omega)$.

Now, replacing $\omega$ by $\theta_{-t}\omega$ in \eqref{eqn:zuz}, noting that $\delta \geq M$  and assuming that $t\geq t_\cD(\omega)$, we compute
\begin{align*}
&\|\nabla v_2^1(t,\theta_{-t}\omega,0)\|_2^2\nonumber\\
&\leq \frac{C_1}{\delta}+C_2\int_0^t \left[\|v_2^1(s,\theta_{-t}\omega, 0)\|_2^2+\|v_1(s,\theta_{-t}\omega,v_1^0(\theta_{-t}\omega))\|_2^2+\|z_1(\theta_{s-t}\omega)\|_2^2\right.\nonumber
\\&\qquad{}\left.+\|\nabla v_1(s,\theta_{-t}\omega,v_1^0(\theta_{-t}\omega))\|_2^2+\|\nabla z_1(\theta_{s-t}\omega)\|_2^2\right] \e{(s-t)M}~\txtd s\nonumber\\
&\leq C_1+C_2\int_0^{t_\cD(\omega)} \left[\|v_2^1(s,\theta_{-t}\omega, 0)\|_2^2+\|v_1(s,\theta_{-t}\omega,v_1^0(\theta_{-t}\omega))\|_2^2\right]\e{(s-t)M}\txtd s\\
&\qquad{}+C_2 \int_{t_\cD(\omega)}^t \left[\|v_2^1(s,\theta_{-t}\omega, 0)\|_2^2+\|v_1(s,\theta_{-t}\omega,v_1^0(\theta_{-t}\omega))\|_2^2\right]\e{(s-t)M}\txtd s\\
&\qquad{} + C_3 \exp(-Mt)(\|v_1^0(\theta_{-t}\omega)\|_2^2+\|v_2^0(\theta_{-t}\omega)\|_2^2)+C_4\int_0^t \exp(M(s-t))\\
&\qquad{} \times (\|z_2(\theta_{s-t}\omega)\|_2^2+\|z_1(\theta_{s-t}\omega)\|_p^p+\|z_1(\theta_{s-t}\omega)\|_2^2+\|\nabla z_1(\theta_{s-t}\omega)\|_2^2)\txtd s\\
&\leq C_1+C_2\left(\|v_1^0(\theta_{-t}\omega)\|^2_{2}+\|v_2^0(\theta_{-t}\omega)\|^2_{2}\right)\e{-Mt} t_\cD(\omega)\\&\qquad{}+ C_5 t_\cD(\omega) \int_0^{t_\cD(\omega)}\e{M(\tau-t)}\left(\|z_2(\theta_{\tau-t}\omega)\|_{2}^2+\|z_1(\theta_{\tau-t}\omega)\|_{p}^p\right)~\txtd \tau\\
&\qquad{} +C_2 \int_{t_\cD(\omega)}^t \rho(\omega)\e{(s-t)M}\txtd s\\
&\qquad{} + C_3 \exp(-Mt)(\|v_1^0(\theta_{-t}\omega)\|_2^2+\|v_2^0(\theta_{-t}\omega)\|_2^2)\nonumber\\
&\qquad{}+C_4\int_{-\infty}^0 \exp(Ms) (\|z_2(\theta_{s}\omega)\|_2^2+\|z_1(\theta_{s}\omega)\|_p^p+\|z_1(\theta_{s}\omega)\|_2^2+\|\nabla z_1(\theta_{s}\omega)\|_2^2)\txtd s\\
&\leq C_1+C_2(t_\cD(\omega))\left(\|v_1^0(\theta_{-t}\omega)\|^2_{2}+\|v_2^0(\theta_{-t}\omega)\|^2_{2}\right)\e{-Mt}+C_3 \rho(\omega)\\
&\qquad +C_4(t_\cD(\omega))\int_{-\infty}^0 \exp(Ms)\\
&\qquad \times (\|z_2(\theta_{s}\omega)\|_2^2+\|z_1(\theta_{s}\omega)\|_p^p+\|z_1(\theta_{s}\omega)\|_2^2+\|\nabla z_1(\theta_{s}\omega)\|_2^2)\txtd s
\end{align*}
where we have used \eqref{eqn:zuzu} in the second inequality and \eqref{eqn:wowo}  in the third inequality. Furthermore, we made use of the absorption property in the third inequality. Finally, since $\|z_2(\theta_{s}\omega)\|_2^2,\|z_1(\theta_{s}\omega)\|_p^p$, $\|z_1(\theta_{s}\omega)\|_2^2,\|\nabla z_1(\theta_{s}\omega)\|_2^2$ (see Lemma \ref{lem:temp} and Remark \ref{rem:nabla}) and $\|v_1^0(\theta_{-t}\omega)\|^2_{2},\|v_2^0(\theta_{-t}\omega)\|^2_{2}$ (by assumption) are tempered random variables, we can combine the right hand side into one tempered random variable $\rho_2(\omega)$ and this concludes the proof.
\end{proof}

\begin{thm} Let Assumptions \ref{ass:1} and \ref{ass:2}  hold. The random dynamical system defined in \eqref{eqn:ourRDS} has a unique $\cT$-random attractor $\cA$.  
\end{thm}
\begin{proof}
By the previous lemmas there exist a compact absorbing set given by \eqref{eqn:cas} in $\mathcal T$ for the RDS $\varphi$. Thus Theorem \ref{thm:atthm} guarantees the existence of a unique $\cT$-random attractor. 
\end{proof}

\section{Applications}
\label{sec:applications}

\subsection{FitzHugh-Nagumo system}

Let us consider the famous stochastic FitzHugh-Nagumo system, i.e., 
\begin{equation}\label{eqn:gm}
\begin{array}{rcl}
\txtd u_1 &=&  \left(\nu_1\Delta u_1-p(x)u_1-u_1(u_1-1)(u_1-\alpha_1)-u_2\right)\txtd t+~B_1\txtd W_1,\\
\txtd u_2 &=&   \left(\alpha_2u_1-\alpha_3u_2\right)\txtd t+~B_2\txtd W_2,
\end{array}
\end{equation}
with $D=[0,1]$ and $\alpha_j\in\R$ for $j\in\{1,2,3\}$ are fixed parameters. We always 
assume that the noise terms satisfy Assumptions~\ref{ass:2} and $p\in C^{2}$. 
Such systems have been considered under various assumptions by numerous authors, for instance see~\cite{BonaMast, Wang1} and the references specified therein. Our general assumptions are satisfied in this example as follows.
Identifying the terms with the terms given in (\ref{eqn:PDE})-(\ref{eqn:ODE}) we have 
\beann
&&h(x,u_1)=p(x)u_1+u_1(u_1-1)(u_1-\alpha_1), \quad f(x,u_1,u_2)=u_2,\\
&&\sigma(x)u_2=\alpha_3 u_2,\quad g(x,u_1)=-\alpha_2u_1.
\eeann
We have $\sigma(x)=\alpha_3$ and $|f(x,u_1,u_2)|=|u_2|$ , i.e., (\ref{eqn:condsi}) and (\ref{eqn:condf}) 
are fulfilled. Furthermore, $|\partial_ug(x,u_1)|=|\alpha_2|$ and $|\partial_{x_i}g(x,u_1)|=0$ for 
$i=1,\ldots,n$, hence (\ref{eqn:condg}) is satisfied. Finally, as a polynomial with odd degree 
and negative coefficient for the highest degree, $h$ fulfils~\eqref{eqn:condh}. 
Thus the analysis above guarantees the existence of global mild solutions and the existence of a 
random pullback attractor for the stochastic FitzHugh-Nagumo system. 

\subsection{The Driven Cubic-Quintic Allen-Cahn Model}

The cubic-quintic Allen-Cahn (or real Ginzburg-Landau) equation is given by
\be
\label{eq:ACcq}
\partial_t u = \Delta u + p_1u+u^3-u^5,\qquad u=u(x,t),
\ee
where $(x,t)\in D\times [0,T)$, $p_1\in\R$, is a fixed parameter and we will take $D$ 
as a bounded open domain with regular boundary.
The cubic-quintic polynomial non-linearity frequently occurs in the modelling of Euler 
buckling~\cite{VenkadesanGuckenheimerValero-Cuevas}, as a re-stabilization mechanism in
paradigmatic models for fluid dynamics~\cite{MorganDawes}, in normal form theory and travelling 
wave dynamics~\cite{KapitulaSandstede,DeisslerBrand}, as well as a test problem for 
deterministic~\cite{KuehnEllipticCont} and stochastic numerical 
continuation~\cite{KuehnSPDEcont}. If we want to allow for time-dependent slowly-varying 
forcing on $u$ and sufficiently regular additive noise, then it is 
actually very natural to extend the model~\eqref{eq:ACcq} 
to 
\be
\label{eq:ACext}
\begin{array}{lcl}
\txtd u_1&=&\left(\Delta u_1 + p_1u_1+u_1^3-u_1^5-u_2\right)~\txtd t+B_1~\txtd W_1,\\
\txtd u_2&=&\varepsilon(p_2 u_2-q_2u_1) ~\txtd t+B_2~\txtd W_2,
\end{array}
\ee
where $p_2$, $q_2$, $0<\varepsilon\ll1 $ are parameters. One easily checks again 
that~\eqref{eq:ACext} fits our general framework as $h(x,u_1)=-p_1u_1-u_1^3+u_1^5$
satisfies the crucial dissipation assumption~\eqref{eqn:condh}. \medskip

\textbf{Acknowledgments:} We thank the anonymous referee for useful comments. CK and AN have been supported by a DFG grant in the D-A-CH framework
(KU 3333/2-1). CK and AP acknowledge support by a Lichtenberg Professorship.



\bibliographystyle{abbrv}
\bibliography{Lit}

\end{document}